\documentclass{article}[11pt]

\usepackage{amsmath,amssymb, graphicx, color}
\usepackage{hyperref}
\usepackage{amsfonts,mathrsfs}
\usepackage{cite}
\usepackage{amsthm}
\usepackage{setspace}
\usepackage{algorithm}
\usepackage[table]{xcolor}
\usepackage{bigdelim}
\usepackage{etoolbox}
\usepackage{blkarray}
\usepackage{qtree}
\usepackage{algpseudocode}
\usepackage{multirow}
\usepackage{bm}
\usepackage{float}
\usepackage[font=small,labelfont=sl]{caption}
\usepackage{tikz}
\usepackage{tabto}
\usepackage[scanall]{psfrag}
\usepackage{anyfontsize}
\usepackage{subfig}

\let\bbordermatrix\bordermatrix
\patchcmd{\bbordermatrix}{8.75}{4.75}{}{}
\patchcmd{\bbordermatrix}{\left(}{\left[}{}{}
\patchcmd{\bbordermatrix}{\right)}{\right]}{}{}

\hypersetup{
pdfpagemode=PageWidth,  
pdftitle={Superfast Tikhonov Regularization of Toeplitz Systems},
colorlinks=true,
linkcolor=blue,
citecolor=red
}

\newcommand{\by}{\times}
\newcommand{\jma}{{\mathrm{\mathbf{j}}}}
\newcommand{\ema}{\mathrm{\mathbf{e}}}
\newcommand{\set}[1]{\left\{#1\right\}}
\newcommand{\dsum}[3]{\displaystyle\sum_{#1}^{#2}{#3}}
\newcommand{\mathtext}[2]{\text{#1}\left({#2}\right)}
\newcommand{\diag}[1]{\mathtext{diag}{#1}}
\newcommand{\degr}[1]{\mathtext{deg}{#1}}
\newcommand{\tdg}[1]{\tau\mbox{--}\degr{#1}}
\newcommand{\eye}{\mathcal{I}}
\newcommand{\eyen}[1]{\eye_{#1}}
\newcommand{\zero}{\mathbf{0}}
\newcommand{\dft}{\mathcal{F}}
\newcommand{\dftn}[1]{\dft_{#1}}
\newcommand{\midx}[1]{\mbox{\scriptsize#1}}
\newcommand{\sci}[2]{#1{\sc e}{-#2}}
\newcommand{\scip}[2]{#1{\sc e}{#2}}
\newcolumntype{C}[1]{>{\centering\let\newline\\\arraybackslash\hspace{0pt}}m{#1}}
\newcommand{\bigo}[1]{\mathcal{O}\left(#1\right)}
\newcommand*\circled[1]{\tikz[baseline=(char.base)]{
            \node[shape=circle,draw,inner sep=1pt] (char) {#1};}}
\newcommand{\faw}[1]{\makebox[1.4em]{$#1$}}
\newcommand{\icarr}[7]{\left[\begin{array}{ccccccc} \faw{#1} & \faw{#2} & \faw{#3} & \faw{#4}
& \faw{#5} & \faw{#6} & \faw{#7} \end{array}\right]}

\newtheorem{theorem}{Theorem}
\newtheorem{definition}{Definition}

\definecolor{lightgray}{gray}{0.90}
\definecolor{lightblue}{rgb}{0.94,0.96,1.0}

\graphicspath{{./figs/}}

\title{Superfast Tikhonov Regularization\\ of Toeplitz Systems}
\author{Christopher~K.~Turnes, Doru~Balcan, Justin~Romberg%
\thanks{C. Turnes and J. Romberg are with the School of Electrical and Computer Engineering at the Georgia Institute of Technology.  
Doru Balcan is with the School of Interactive Computing at the Georgia Institute of Technology.
This work was partially supported by a Packard Fellowship.}}

\begin{document}
\maketitle

\begin{abstract}
Toeplitz-structured linear systems arise often in practical engineering problems.  Correspondingly, a number of algorithms 
have been developed that exploit Toeplitz structure to gain computational efficiency when solving these systems.  The earliest
``fast'' algorithms for Toeplitz systems required $\mathcal{O}(n^2)$ operations, while more recent ``superfast'' algorithms 
reduce the cost to $\mathcal{O}(n\log^2 n)$ or below.

In this work, we present a superfast algorithm for Tikhonov regularization of Toeplitz systems.  
Using an ``extension-and-transformation'' technique, our algorithm translates a Tikhonov-regularized Toeplitz system 
into a type of specialized polynomial problem known as tangential interpolation.  
Under this formulation, we can compute the solution in only $\mathcal{O}(n\log^2 n)$ operations.
We use numerical simulations to demonstrate our algorithm's complexity and verify that it returns 
stable solutions. 
\end{abstract}

\section{Introduction}
\label{sec:intro}

This paper develops a computationally efficient and numerically stable algorithm for solving systems of equations with {\em Toeplitz} structure.  Toeplitz matrices, which arise in problems involving temporally- or spatially-invariant systems, have constant diagonal coefficients $T=[a_{i-j}]$.  This type of structure can be exploited to accelerate the calculation of least-squares solutions, a strategy first adopted in the mid 1900s when Levinson \cite{Levinson:1947}, Durbin \cite{Durbin:1960}, and others \cite{Trench:1964} \cite{Zohar:1969} \cite{Bareiss:1969} developed a series of inversion algorithms that reduced the cost of solving an $n\by n$ Toeplitz system from $\bigo{n^3}$ to only $\bigo{n^2}$ operations.  Algorithms like Levinson's are called {\sl fast}, and they have proven exceedingly useful in a wide variety of signal-processing applications.  A concise but comprehensive review of fast Toeplitz solvers is available in~\cite{Heinig:2011}.  

While these algorithms have small overhead costs and are well-suited for small- to medium-size problems, they require a number of operations an order higher than the number of free parameters in the systems they solve.  Over the past few decades, new algorithms have bridged this gap by inverting Toeplitz matrices in strictly less than $\mathcal{O}(n^2)$ operations (we note in particular \cite{Ammar:1988}, \cite{Barel:2001}, and \cite{Chandrasekaran:2007}), earning the designation ``{\sl superfast}.''  While usually more complicated than their predecessors, superfast algorithms can provide an enormous reduction in computation time for large or very large matrices.

Unfortunately, inversion is rarely practical; the matrices that arise in applications are seldom square and nonsingular.  As a result, many Toeplitz inversion methods have been modified to solve more general least-squares problems.  For example, a superfast pseudoinversion algorithm was given in~\cite{Barel:2003}, where the authors extended the algorithm of~\cite{Barel:2001} to solve rectangular Toeplitz systems with full column rank.  Their extended algorithm applies the pseudoinverse in $\bigo{N\log^2 N}$ operations, where $N$ is the number of free parameters defining the system.  Of course, this is not an isolated development; other notable Toeplitz least-squares solvers include \cite{Heinig:2004}, \cite{Sweet:1984}, \cite{Bini:2003}, and \cite{Pan:2004}, to name a few.

In this work, we further adapt the approach of \cite{Barel:2003} for Tikhonov regularization of Toeplitz systems.  Our results are not unprecedented; a superfast Tikhonov solver based on displacement structure (first introduced in \cite{Kailath:1979}) was presented in \cite{Kimitei:2008}. However, ours takes a fundamentally different approach: we reformulate the system as an interpolation problem instead of using its structure to accelerate matrix factorization. This strategy allows our algorithm to be more general in scope.  Whereas the approach of~\cite{Kimitei:2008} applies only when the regularization penalizes solutions of large Euclidean norm, our algorithm is applicable for any Toeplitz regularizer, easily extends to include a number of different regularization terms, and applies when the Gramian $G_T=T^HT$, but not the matrix $T$ itself, is Toeplitz. 

The remainder of the paper is organized as follows.  In Section~\ref{sec:prob}, we formulate the Tikhonov-regularization problem in a general sense and discuss the specific cases our algorithm addresses.  To conclude the section, we outline our two-part ``extension-and-transformation'' approach.  Subsequently, we develop the first portion of this approach by reframing the regularization problems as partial-circulant-block systems in Section~\ref{sec:extend}.  In Section~\ref{sec:transsolve}, we use Fourier operators to transform these systems into  polynomial-interpolation problems and explain how such problems may be solved efficiently.  We then summarize our algorithm in Section~\ref{sec:algorithm} and discuss implementation issues in Section~\ref{sec:implementation}.  In Section~\ref{sec:simulations}, we give the results of numerical simulations verifying the cost and utility of our algorithm.  Finally, we summarize our developments and detail potential future extensions in Section~\ref{sec:conclusions}.

We have implemented our algorithm as part of a comprehensive code base written in \textsc{MATLAB} and \textsc{C++}.  The code package can be obtained online\footnote{The code package is available at \url{http://users.ece.gatech.edu/~cturnes3}.}, along with scripts that reproduce all of the numerical experiments in Section~\ref{sec:simulations}.

\section{Problem formulation}
\label{sec:prob}

When a system of equations $Tx=b$ is ill-posed, either because the solution $x$ is non-unique or because it is non-existent, a least-squares solution $\hat{x}$ is usually calculated instead.  The solution $\hat{x}$ is designed to minimize the norm of the residual vector $r=b-T\hat{x}$, but might have undesirable properties depending on the characteristics of the matrix $T$. A common technique to avoid this problem is to regularize the system by including a term that penalizes solutions that are not well-behaved.  This approach is known as {\sl Tikhonov regularization}\footnote{The term ``ridge regression'' is used in the field of statistics.}, and is formulated by expressing $\hat{x}$ as the minimizer of the optimization problem
\begin{equation}
\label{eqn:tik}
\hat{x} = \underset{x}{\arg\min}\hspace{5pt}\|Tx-b\|^2+\|Lx\|^2.
\end{equation}
The matrix $L$, which provides the regularization, is referred to as the {\sl Tikhonov matrix} or the {\sl regularizer}.  If a solution to~\eqref{eqn:tik} exists, it is expressed in closed form as
\begin{equation}
\label{eqn:tiksol}
\begin{split}
\hat{x} &= (T^HT+L^HL)^{-1}T^Hb \\
&= (G_T+G_L)^{-1}T^Hb,
\end{split}
\end{equation}
where $G_T$ and $G_L$ are the Gramians of $T$ and $L$, respectively.

In this work, we exploit Toeplitz structure to reduce the number of operations necessary to compute~\eqref{eqn:tiksol}.  Throughout the exposition, we consider three specific scenarios:
\begin{itemize}
\item {\bf General problem}:  $T$ and $L$ are $m\by n$ and $p\by n$ Toeplitz.

\item {\bf $\ell_2$-norm penalization}:  $T$ is $m\by n$ Toeplitz while $L$ is a scaled identity matrix.

\item {\bf Toeplitz-Gramian problem}:  $G_T=T^HT$ is $n\by n$ Hermitian Toeplitz while $L$ is $p\by n$ Toeplitz.

\end{itemize}
In fact, our approach can be applied to a much larger class of problems as well; these three simply represent those cases that appear to be most interesting.

To solve these problems, we present an algorithm that calculates~\eqref{eqn:tiksol} with superfast complexity $\bigo{N\log^2 N}$, where $N$ is the number of free parameters.  Our algorithm can be divided into three stages.  First, we decouple the Gramian matrices $G_T$ and $G_L$ to express the problem as a Toeplitz-block system.\footnote{For the Toeplitz-Gramian problem, we need only decouple the matrix $G_L$.} Next, we follow the ``extension-and-transformation'' approach of~\cite{Heinig:1996} to translate our linear-algebraic problems into the context of a polynomial problem known as {\sl tangential interpolation}~\cite{Fechina:1975}.  Finally, using strategies similar to those of~\cite{Barel:2003}, we solve the tangential-interpolation problems in $\bigo{N\log^2 N}$ to compute the solutions $\hat{x}$.  

\section{Extensions of Tikhonov systems}
\label{sec:extend}

Our algorithm reformulates linear-algebraic problems as polynomial problems. The first step in achieving this change of context is to expand the linear systems by adding block rows and columns.  We motivate this process by demonstrating how a Toeplitz system may be extended to form an equivalent partial-circulant system.  We then use this technique to extend the Tikhonov systems of the three problems, expressing them as partial-circulant-block systems.  As a point of notation, we will indicate block row and column sizes outside of matrix brackets to better communicate submatrix sizes.

\subsection{Circulant extensions of Toeplitz matrices}
\label{ssec:circextend}

In this section, we present a simple extension of Toeplitz matrices into circulant submatrices.  This technique can be used to replace a Toeplitz problem with a larger circulant-like system, but will not add or remove any information.  Instead, it will yield a form that is more easily manipulable.

Consider the linear system $Tx=b$, where $T=[a_{i-j}]$ is $m\by n$ Toeplitz.  Our objective is to find a system with circulant structure that we might solve in place of the original problem. For any $k\geq 0$, we define an integer $N_k:=m+n+k-1$.  If $\overline{T}$ is the $(n+k-1)\by n$ Toeplitz matrix 
\[ 
\overline{T} = 
\begin{bmatrix}
a_{-n-k} & a_{m-1} & \cdots & a_{m-n+1} \\
a_{-n-k+1} & a_{-n-k} & \cdots & a_{m-n+2} \\
\vdots & \vdots & \ddots & \vdots \\
a_{-1} & a_{-2} & \cdots & a_{-n}
\end{bmatrix},
\]
where the coefficients $a_i$ are arbitrary for $i\leq -n$, then the $N_k\by n$ matrix
\begin{equation}
\label{eqn:circexp}
C_k(T) = \begin{blockarray}{c@{\hspace{5pt}}r@{\hspace{5pt}}cl}
 & \midx{n} & & \\
 \begin{block}{[c@{\hspace{5pt}}r@{\hspace{5pt}}c]l}
 & \overline{T} & & \midx{n+k-1} \\
 & T & & \midx{m} \\
 \end{block}
\end{blockarray}
\end{equation}
consists of the first $n$ columns of the $N_k\by N_k$ circulant matrix
\[ \mathrm{circ}_k(T)=\begin{bmatrix}
a_{-n-k} & \cdots &  a_{-n-k+1}\\
\vdots & \ddots & \vdots \\
a_{m-1} & \cdots & a_{-n-k}
\end{bmatrix}.\]  
We call $\overline{T}$ the {\sl extension matrix} of $T$, while $C_k(T)$ is the {\sl $k$-circulant extension of $T$}.

The matrix $T$ in the original linear system cannot simply be replaced by its circulant extension $C_k(T)$, as this would introduce extra rows to the system without accounting for them on the right-hand side of the equation. Instead, we must also add an equal number of columns, which we generate by defining an artificial unknown $\gamma=-\overline{T}x$.\footnote{It is ``artificial'' because it adds no new information to the system.}  Using $\gamma$, if $\eyen{n+k-1}$ is the $(n+k-1)\by (n+k-1)$ identity matrix, the linear system $Tx=b$ is equivalent to
\begin{equation}
\label{eqn:expand1}
\bbordermatrix{
 & \midx{n} & \midx{n+k-1} \cr
\midx{n+k-1}\quad & \overline{T} & \eyen{n+k-1} \cr
\midx{m} & T & \zero \cr
}
\begin{bmatrix}
x \\ \gamma 
\end{bmatrix} = 
\begin{bmatrix} 
\zero \\ b 
\end{bmatrix}.
\end{equation}
The expansion adds just as many columns as rows, so the new system is as over- or underdetermined as the original.

The blocks of~\eqref{eqn:expand1} can be grouped into partial-circulant matrices.  Let the matrix subscript $\Gamma_\ell$ represent the submatrix formed from the first $\ell$ columns of a matrix.  For example, the $k$-circulant extension $C_k(T)$ can be written as $\left(\mathrm{circ}_k(T)\right)_{\Gamma_n}$.  Using this notation, and grouping the block rows of~\eqref{eqn:expand1} together, each block column is a circulant submatrix:
\begin{equation}
\label{eqn:expand2}
\begin{bmatrix} C_k(T) & \left(\eyen{N_k}\right)_{\Gamma_{n+k-1}} \end{bmatrix}
\begin{bmatrix} 
x \\ \gamma \end{bmatrix} = 
\begin{bmatrix} \zero \\ b \end{bmatrix}. \nonumber
\end{equation}

This grouping is advantageous; circulant matrices are easier to manipulate than Toeplitz matrices, as they are diagonalized by the Fourier matrix.  Namely, if $\dftn{N_k}$ is the $N_k\by N_k$ Fourier matrix, the matrix $C_k(T)$ may be decomposed as 
\begin{equation}
\label{eqn:circdiag} 
C_k(T) = \left(\dftn{N_k}^H\Lambda\dftn{N_k}\right)_{\Gamma_n}=\dftn{N_k}^H\Lambda\left(\dftn{N_k}\right)_{\Gamma_n}
\end{equation}
for some $N_k\by N_k$ diagonal matrix $\Lambda$. Of course, there exists a similar decomposition for the block $\left(\eyen{N_k}\right)_{\Gamma_{n+k}}$ as well.  To avoid compounding subscripts as in~\eqref{eqn:circdiag}, we will omit those indicating the size of the Fourier matrices, as they can be easily inferred from context.  For instance, we express~\eqref{eqn:circdiag} as 
\[ C_k(T) = \dft^H\Lambda\dft_{\Gamma_n}.\]

\subsection{Extension for the general problem}
\label{ssec:toeptik}

We now apply the technique of Section~\ref{ssec:circextend} to the Tikhonov system in~\eqref{eqn:tiksol} for the general problem.  First, we define two artificial variables $\sigma_1:=T\hat{x}$ and $\sigma_2:=L\hat{x}$, which allow us to decouple the Gramian matrices:
\begin{eqnarray*}
 T^Hb  &=& T^H(Tx)+L^H(Lx)\\
&=& T^H\sigma_1 + L^H\sigma_2.
\end{eqnarray*}
Using these variable definitions,~\eqref{eqn:tiksol} may be replaced by a system containing only Toeplitz blocks:
\begin{equation}
\label{eqn:ex1}
\begin{bmatrix}
\zero & T^H  & L^H \\
-T & \eyen{m} & \zero \\
 -L & \zero & \eyen{p}\\
\end{bmatrix}\begin{bmatrix}
\hat{x} \\
\sigma_1 \\
\sigma_2 \\
\end{bmatrix} = \begin{bmatrix} T^Hb \\ \zero \\ \zero \end{bmatrix}.
\end{equation}
\noindent Since all of the blocks of this system are Toeplitz, we can extend them as in~\eqref{eqn:expand1}, using $1$-circulant extensions for simplicity.  Since the extension size must be large enough that all of the blocks become circulant submatrices, the first block row of the extended system will have $q+n$ total rows, where $q=\max\left(m,p\right)$. 

If we are to add the extensions $\overline{T^H}$, $\overline{L^H}$, $\overline{T}$, and $\overline{L}$ to our system, we must introduce three corresponding artificial variables to compensate for the extra rows:
\begin{align*}
\gamma_1 &:= -\overline{T^H}\sigma_1-\overline{L^H}\sigma_2, &
\gamma_2 &:= \overline{T}\hat{x}, &
\gamma_3 &:= \overline{L}\hat{x}.
\end{align*}
With these variables defined, and letting $\hat{\bm{b}}=-T^Hb$, the system of~\eqref{eqn:ex1} has the equivalent form
\begin{equation}
\edef\savedbaselineskip{\the\baselineskip\relax}
\begin{array}{ccl}
\baselineskip=\savedbaselineskip
\bbordermatrix{
~ & \midx{n} & \midx{m} & \midx{p} & \midx{q} & \midx{n} & \midx{n} & \midx{1} \cr
\midx{q} &  \zero & \overline{T^H} & \overline{L^H} & \eyen{q} & \zero & \zero & \zero\cr
\midx{n} &  \zero & T^H &L^H & \zero & \zero & \zero & \hat{\bm{b}} \cr
\midx{n} &  -\overline{T} & \zero &\zero & \zero & \eyen{n} & \zero & \zero   \cr
\midx{m} &  -T & \eyen{m} &\zero & \zero & \zero & \zero & \zero \cr
\midx{n} &  -\overline{L} & \zero & \zero & \zero & \zero & \eyen{n} & \zero \cr
\midx{p} &  -L & \zero &\eyen{p} & \zero & \zero & \zero & \zero
} & \begin{bmatrix}
\hat{x} \\
\sigma_1 \\
\sigma_2 \\
\gamma_1 \\
\gamma_2 \\
\gamma_3 \\
1
\end{bmatrix} & =\zero\vspace{5pt} \\ 
\bm{C} & \bm{p}^* &= \zero,
\end{array}
\label{eqn:ex2}
\end{equation}
where we have moved the right-hand side vector into the matrix to generate a homogeneous system of equations.  We have also expressed the unknown variables as a single block column vector $\bm{p}^*$ to simplify future expressions. 

Grouping the blocks of the matrix $\bm{C}$ together to form circulant submatrices, we may replace the original system of~\eqref{eqn:tiksol} with the much larger system
\begin{equation}
\label{eqn:ex3}
\bbordermatrix{
~ & \midx{n} & \midx{m} & \midx{p} & \midx{q} & \midx{n} & \midx{n} & \midx{1} \cr
\midx{q+n} &  \zero & C_{12} & C_{13} & C_{14} & \zero & \zero & C_{17} \cr
\midx{m+n} & C_{21} & C_{22} & \zero & \zero & C_{25} & \zero & \zero \cr
\midx{p+n} &  C_{31} & \zero & C_{33} & \zero & \zero & C_{36} & \zero \cr
}\bm{p}^*=\zero.
\end{equation}
The $\set{C_{ij}}$ are circulant submatrices, and can be factored with Fourier matrices of the appropriate sizes.

\subsection{Extension for $\ell_2$-norm penalization}
\label{ssec:toepeuc}

The system in~\eqref{eqn:ex3} may be simplified if the Tikhonov matrix is a scaled identity matrix ({\sl i.e.}, if $L=\beta \eyen{n}$ for some $\beta\in\mathbb{C}$).  Specifically, it may be drastically reduced in size by observing that the extension matrix for $L$ is $\zero$ if all arbitrary entries are set to zero.  Using this fact, the extended system for $\ell_2$-norm penalization is given by
\begin{equation}
\label{eqn:ex5}
\bbordermatrix{
~ & \midx{n} & \midx{m} & \midx{n} & \midx{q} & \midx{n} & \midx{n} & \midx{1} \cr
\midx{q} &  \zero & \overline{T^H} & \zero & \eyen{q} & \zero & \zero & \zero\cr
\midx{n} &  \zero & T^H & \beta^*\eye_n & \zero & \zero & \zero & \hat{\bm{b}} \cr
\midx{n} &  -\overline{T} & \zero &\zero & \zero & \eyen{n} & \zero & \zero   \cr
\midx{m} &  -T & \eyen{m} &\zero & \zero & \zero & \zero & \zero \cr
\midx{n} & \zero & \zero & \zero & \zero & \zero & \eyen{n} & \zero \cr
\midx{n} & -\beta\eyen{n} & \zero &  \eyen{n} & \zero & \zero & \zero & \zero \cr
}\begin{bmatrix}
\hat{x} \\
\sigma_1 \\
\sigma_2 \\
\gamma_1 \\
\gamma_2 \\
\gamma_3 \\
1
\end{bmatrix}=\zero,
\end{equation}
where $\beta^*$ is the complex conjugate of $\beta$.  

The fifth block row of~\eqref{eqn:ex5} implies $\gamma_3=\zero$, while the sixth implies $\beta\hat{x}=\sigma_2$. As a result, we can eliminate both $\gamma_3$ and $\sigma_2$ as variables. Since $\sigma_2$ has been removed, the only matrix with a non-zero extension in the first block row is $T^H$, and therefore we set $q=m$ regardless of whether or not $m<n$.  

After reducing the number of unknowns, we are left with the following simplified system:
\[
\bbordermatrix{
~ & \midx{n} & \midx{m} & \midx{m} & \midx{n} & \midx{1} \cr
\midx{m} & \zero & \overline{T^H} & \eyen{m} & \zero & \zero \cr
\midx{n} & |\beta|^2\eyen{n}& T^H & \zero & \zero  & \hat{\bm{b}}\cr
\midx{n} &  -\overline{T} & \zero & \zero & \eyen{n} & \zero \cr
\midx{m} &  -T &  \eyen{m} & \zero & \zero & \zero \cr
}\begin{bmatrix}
\hat{x} \\
\sigma_1 \\
\gamma_1 \\
\gamma_2 \\
1
\end{bmatrix}=\bm{C}\bm{p}^*=\zero.
\]
Again grouping the Toeplitz components together with their extension matrices, we can express this as a system with partial-circulant blocks:
\begin{equation}
\label{eqn:ex6} 
\bbordermatrix{
~ & \midx{n} & \midx{m} & \midx{m} & \midx{n} & \midx{1} \cr
\midx{m+n} & C_{11} & C_{12} & C_{13} & \zero & C_{15} \cr
\midx{m+n} & C_{21} & C_{22} & \zero & C_{24} & \zero \cr
}\bm{p}^*=\zero.
\end{equation}
This system is significantly smaller than that of~\eqref{eqn:ex3}, and it can be solved more efficiently.

\subsection{Extension for the Toeplitz-Gramian problem}
\label{ssec:toepgram}

If the Gramian matrix $G_T=T^HT$ is Toeplitz, the Tikhonov system can be extended in a manner similar to the developments of Section~\ref{ssec:toeptik}. However, this is not always a wise approach; if $L$ is a scaled identity matrix, the matrix $(G_T+L^HL)$ is also Toeplitz, meaning the problem may be solved by the direct-inversion algorithm of~\cite{Barel:2001} (or numerous alternatives). Therefore, when we consider the Toeplitz-Gramian problem, we assume without loss of generality that $L$ is a generic $p\by n$ Toeplitz matrix.

For this problem, we need only introduce a single artificial variable $\sigma=Lx$.  Using the two equations
\begin{eqnarray*}
G_Tx + L^H\sigma &=& T^Hb,\cr
-Lx + \sigma &=& 0,
\end{eqnarray*}
we can perform $1$-circulant extensions on $G_T$, $L^H$, and $L$ to obtain the block system
\begin{equation}
\label{eqn:gramblk}
\bbordermatrix{
~ & \midx{n} & \midx{p} & \midx{q} & \midx{n} & \midx{1} \cr
\midx{q} & \overline{G} & \overline{L^H} & \eyen{q} & \zero & \zero \cr
\midx{n} & G & L^H & \zero & \zero & \hat{\bm{b}} \cr
\midx{n} & -\overline{L} & \zero & \zero & \eyen{n} & \zero \cr
\midx{p} & -L & \eyen{p} &\zero & \zero & \zero \cr
}\begin{bmatrix}
x \\
\sigma \\
\gamma_1 \\
\gamma_2 \\
1
\end{bmatrix}=\bm{C}\bm{p}^*=\zero. \nonumber
\end{equation}
Since we are working with the Gramian of $T$ rather than $T$ itself, $q=\max\left(n,p\right)$ .

Grouping the component blocks together yields the partial-circulant-block system:
\begin{equation}
\label{eqn:ex7} 
\bbordermatrix{
~ & \midx{n} & \midx{p} & \midx{q} & \midx{n} & \midx{1} \cr
\midx{q+n} & C_{11} & C_{12} & C_{13} & \zero & C_{15} \cr
\midx{p+n} & C_{21} & C_{22} & \zero & C_{24} & \zero \cr
}\bm{p}^*=\zero.
\end{equation}
This is an identical formulation to the $\ell_2$-norm-penalization problem, but with the circulant blocks $\set{C_{ij}}$ defined differently (and with different matrix sizes).  Again, compared to the system of~\eqref{eqn:ex3}, the system of~\eqref{eqn:ex7} is significantly smaller and simpler to solve.

\section{Transforming and solving block circulant systems}
\label{sec:transsolve}

In this section, we detail the ``transformation'' part of our approach, using Fourier operators to translate the systems of~\eqref{eqn:ex3},~\eqref{eqn:ex6}, and~\eqref{eqn:ex7} into the language of polynomials.  This new formulation allows us to solve interpolation problems rather than linear-algebraic problems, and is based on the framework of~\cite{Barel:1992}.  Our approach is not confined to the examples we consider; any nonsingular system having only partial-circulant blocks may be solved in a similar way.

We consider only the general problem for the rest of the analysis, as the requisite adaptations for $\ell_2$-norm penalization and the Toeplitz-Gramian problem are straightforward.  While there are fewer interpolation conditions for the latter two problems, their solution methods are virtually identical and all theoretical results extend easily.  We also assume that $m=p=q$ for simplicity, though this assumption is not fundamental to our results.

\subsection{Transformation}
\label{ssec:transform}

After extending the systems, we place them in the context of polynomials by applying Fourier operators.  Let $N=m+n$ be the number of rows in each block row of~\eqref{eqn:ex3}.   Using the diagonal decomposition of~\eqref{eqn:circdiag}, we can left-multiply by the $3N\by 3N$ Fourier-block operator
\[ \eye_3\otimes\mathcal{F}=\begin{bmatrix}
\dft & \zero & \zero \\
\zero & \dft & \zero \\
\zero & \zero & \dft
\end{bmatrix},
\]
transforming the homogeneous system into the form
\begin{equation}
\label{eqn:ex4}
\bbordermatrix{
~ & \midx{N} & \midx{N} & \midx{N} & \midx{N} & \midx{N} & \midx{N} & \midx{N} \cr
\midx{N} &  \zero & \Lambda_{12} & \Lambda_{13} & \Lambda_{14} & \zero & \zero & \Lambda_{17} \cr
\midx{N} & \Lambda_{21} & \Lambda_{22} & \zero & \zero & \Lambda_{25} & \zero & \zero \cr
\midx{N} & \Lambda_{31} & \zero & \Lambda_{33} & \zero & \zero & \Lambda_{36} & \zero \cr
}\bm{F}\bm{p}^*=\zero,
\end{equation}
where $\Lambda_{ij}=\text{diag}(\lambda_{ij}^{(0)},\ldots,\lambda_{ij}^{(N-1)})$ is the diagonal matrix in the factorization of $C_{ij}$ and $\bm{F}$ is the block-diagonal matrix
\[ \bm{F} = \diag{\mathcal{F}_{\Gamma_n},\mathcal{F}_{\Gamma_m},\mathcal{F}_{\Gamma_m},\mathcal{F}_{\Gamma_m},\mathcal{F}_{\Gamma_n},\mathcal{F}_{\Gamma_n},\mathcal{F}_{\Gamma_1}}.\]

The transformation allows us to replace~\eqref{eqn:ex3} with a set of polynomial interpolation conditions.  The key to such a shift in perspective is to equate a vector $\rho=[\rho_i]$ with a polynomial $\rho(z)=\sum_i \rho_iz^{i-1}$. Under this equivalence, the product
\begin{equation}
\mathcal{F}\rho = \begin{bmatrix}
\omega_0^0 & \cdots & \omega_0^{N-1} \\
 \vdots & \ddots & \vdots \\
\omega_{N-1}^0 & \cdots & \omega_{N-1}^{N-1}
\end{bmatrix}\begin{bmatrix}
\rho_0 \\ \vdots \\ \rho_{N-1} \end{bmatrix}=\begin{bmatrix}
\rho(\omega_0) \\
\vdots \\
\rho(\omega_{N-1})
\end{bmatrix},
\label{eqn:vandprod}
\end{equation}
where $\omega_k=\ema^{\jma 2\pi k/N}$, produces a vector of the polynomial evaluations $[\rho(\omega_i)]$.

The components of $\bm{p}^*$ in~\eqref{eqn:ex4} correspond to polynomials $\set{\hat{x}(z),\sigma_1(z),\sigma_2(z),\gamma_1(z),\gamma_2(z),\gamma_3(z),1}$.  Applying the equivalence of~\eqref{eqn:vandprod}, the product $\bm{F}\bm{p}^*$ may be treated as an unknown vector containing the values of these polynomials at the nodes $\omega_k$.  Since the diagonal matrices $\Lambda_{ij}$ scale these values,~\eqref{eqn:ex4} is equivalent to a set of $3N$ interpolation conditions:
\begin{equation}
\begin{aligned}
&\lambda_{12}^{(k)}\sigma_1(\omega_k)+\lambda_{13}^{(k)}\sigma_2(\omega_k)+\lambda_{14}^{(k)}\gamma_1(\omega_k)+\lambda_{17}^{(k)} = 0, \\
&\lambda_{21}^{(k)}\hat{x}(\omega_k)+\lambda_{22}^{(k)}\sigma_1(\omega_k)+\lambda_{25}^{(k)}\gamma_2(\omega_k) = 0,\hspace{4pt}\text{and}\\
&\lambda_{31}^{(k)}\hat{x}(\omega_k)+\lambda_{33}^{(k)}\sigma_2(\omega_k)+\lambda_{36}^{(k)}\gamma_3(\omega_k) = 0
\end{aligned}
\label{eqn:intcond}
\end{equation}
for $k=0,\ldots,N-1$.  
These equations are markedly different from those of a standard interpolation problem.  Rather than prescribing individual polynomial values at each node, the conditions define values of weighted sums of polynomials. 

To solve this type of problem, we gather the unknown polynomials into a vector polynomial $p^*(z)\in\mathbb{C}[z]^{7\by 1}$. Defining the interpolation conditions
\begin{align*}
F_k &= \icarr{0}{\lambda_{1,2}^{(k)}}{\lambda_{1,3}^{(k)}}{\lambda_{1,4}^{(k)}}{0}{0}{\lambda_{1,7}^{(k)}} \\
G_k &= \icarr{\lambda_{2,1}^{(k)}}{\lambda_{2,2}^{(k)}}{0}{0}{\lambda_{2,5}^{(k)}}{0}{0},\text{ and } \\
H_k &= \icarr{\lambda_{3,1}^{(k)}}{0}{\lambda_{3,3}^{(k)}}{0}{0}{\lambda_{3,6}^{(k)}}{0},
\end{align*}
the solution $p^*(z)$ has (component-wise) degree
\begin{equation} 
\label{eqn:degcon} 
\degr{p^*} < \begin{bmatrix} n & m & m & m & n &  n & 1 \end{bmatrix}^T 
\end{equation}
and satisfies
\begin{equation} 
\label{eqn:interp}
F_kp^*(\omega_k)=G_kp^*(\omega_k)=H_kp^*(\omega_k)=0\end{equation}
for $k=0,\ldots,N-1$.  Equations in the form of~\eqref{eqn:interp} are known as {\sl tangential-interpolation} conditions.

\subsection{Tangential interpolation}
\label{ssec:frameor}

After transformation, we can compute~\eqref{eqn:tiksol} by tangentially interpolating a polynomial vector from the conditions in~\eqref{eqn:interp}.  The interpolation can be calculated efficiently with the algorithm first described in~\cite{Barel:1994} and later improved in~\cite{Barel:1998}. To introduce this method, we present several components of the theoretical framework of~\cite{Barel:1992}. As our goal for this section is to provide an overview, we omit many technical details.

We begin by defining an algebraic context for the problem. The solution $p^*(z)$ is an element of $\mathbb{C}[z]^{7\by 1}$, the space of all $7\by 1$ vector polynomials with complex coefficients.  The set $\mathbb{C}[z]^{7\by 1}$ is a {\sl module}, a more abstract form of a vector space.  It is similar in nature to the vector space $\mathbb{C}^{7\by 1}$, but its elements are vectors of complex-valued polynomials.

The vector polynomial $p^*(z)$ is not an arbitrary element of $\mathbb{C}[z]^{7\by 1}$; it also satisfies the interpolation conditions in~\eqref{eqn:intcond}. We therefore limit our search for $p^*(z)$ to the much smaller set of elements with this property. To formally establish this set, we define the {\sl tangential-interpolation residual}.  
\begin{definition}[Tangential-interpolation residual]
\label{def:residual}
The residual of a vector polynomial $q(z)\in\mathbb{C}[z]^{7\by 1}$ relative to the interpolation conditions $\set{F_k,G_k,H_k}$ is 
\[ (r(q))_k\equiv\begin{bmatrix}
F_k \\
G_k \\
H_k \end{bmatrix}q(\omega_k).
\]
\end{definition}
\noindent With this definition, the set of all vector polynomials satisfying the interpolation conditions is given as
\[ \mathcal{S} := \set{\ p(z)\in\mathbb{C}[z]^{7\by 1}\ :\ \left(r(p)\right)_k=\zero\hspace{10pt}\forall k}. \] 
This set is a {\sl submodule} of $\mathbb{C}[z]^{7\by 1}$, and characterizes the null space of the matrix in~\eqref{eqn:ex4}. 

The linear equations of~\eqref{eqn:ex4} form an underdetermined homogeneous system, and have an infinite number of solutions. Correspondingly, the submodule $\mathcal{S}$ is infinite, and not all of its elements are related to the solution $p^*(z)$.  For example, $(z^N-1)q(z)\in\mathcal{S}$ for {\sl any} element $q(z)\in\mathbb{C}[z]^{7\by 1}$. To compute the interpolation, then, we need to differentiate $p^*(z)$ from the irrelevant elements of $\mathcal{S}$. 

Fortunately, the known degree structure of $p^*(z)$ sets it apart.  Since the vector $\bm{p}^*$ is the {\sl unique} solution to~\eqref{eqn:ex2}, the only elements of $\mathcal{S}$ that satisfy~\eqref{eqn:degcon} are those of the form $\alpha p^*(z)$, where $\alpha\in\mathbb{C}$.  As a result, if we can find an element of $\mathcal{S}$ with the proper degree structure, the solution $p^*(z)$ may be calculated with simple scaling.

To analyze the degree structure of the elements of $\mathcal{S}$, we use a tool known as the {\sl $\tau$-degree}~\cite{Barel:1992}.\footnote{In its original form (in~\cite{Barel:1992}), the $\tau$-degree was referred to as the $\vec{s}$-degree. This terminology was later changed in~\cite{Barel:2001} and~\cite{Barel:2003}.}
\begin{definition}[$\tau$-degree]
\label{def:taudeg}
Let $\tau\in\mathbb{Z}^7$.  The $\tau$-degree of a vector polynomial $p(z)\in\mathbb{C}[z]^{7\by 1}$ is the integer $\delta\in\mathbb{Z}$ given by
\begin{equation}
\label{eqn:taudeg}
\delta=\tdg{p} = \underset{i}{\max}\left(\degr{p_i}-\tau_i\right),
\end{equation}
where $\degr{0}=-\infty$.  
\end{definition}
\noindent The $\tau$-degree is the maximum polynomial degree in a vector polynomial after each of its components have been ``shifted'' by some set amount.   It is represented visually in Fig.~\ref{fig:taudeg}. For different choices of $\tau$, the $\tau$-degree may be different (and the components determining the $\tau$-degree may vary as well).

\begin{figure}[!t]
\centering
\begin{tabular}{ccc}
\psfrag{Component}{\raisebox{-3pt}{\fontsize{10}{11.2}\selectfont Component}}
\psfrag{deg pi}{\raisebox{3pt}{\fontsize{10}{11.2}\selectfont $\degr{p_i}$}}
\includegraphics[width=0.3\columnwidth]{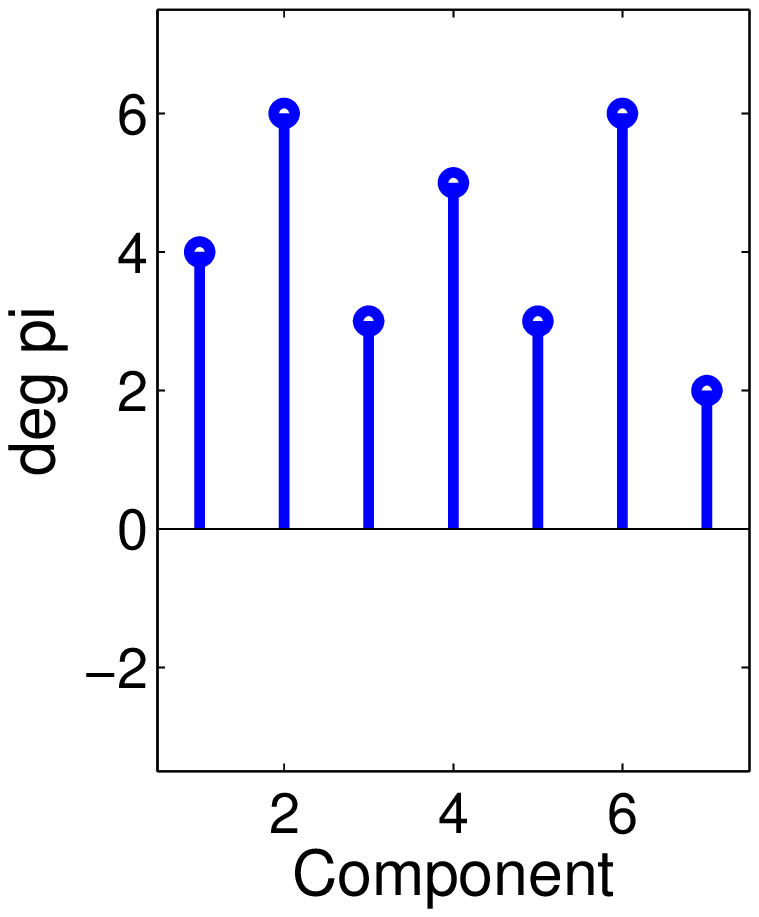} &
\psfrag{Component}{\raisebox{-3pt}{\fontsize{10}{11.2}\selectfont Component}}
\psfrag{taui}{\raisebox{3pt}{\fontsize{10}{11.2}\selectfont $\tau_i$}} 
\includegraphics[width=0.3\columnwidth]{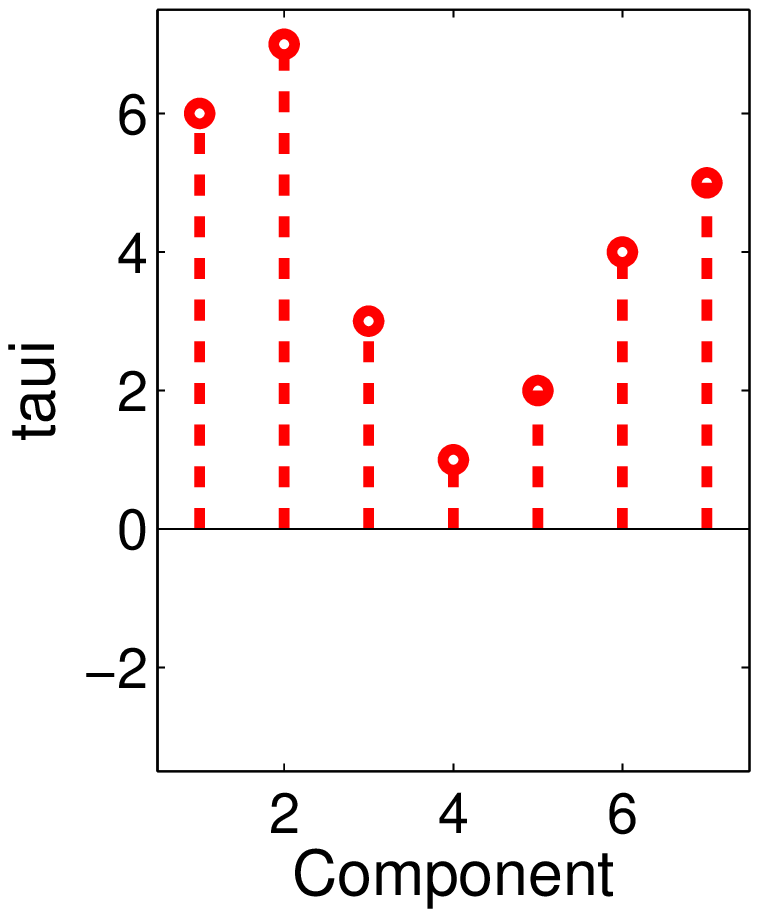} & 
\psfrag{Component}{\raisebox{-3pt}{\fontsize{10}{11.2}\selectfont Component}}
\psfrag{deg pi - taui}{\raisebox{3pt}{\fontsize{10}{11.2}\selectfont $\degr{p_i}-\tau_i$}}
\includegraphics[width=0.3\columnwidth]{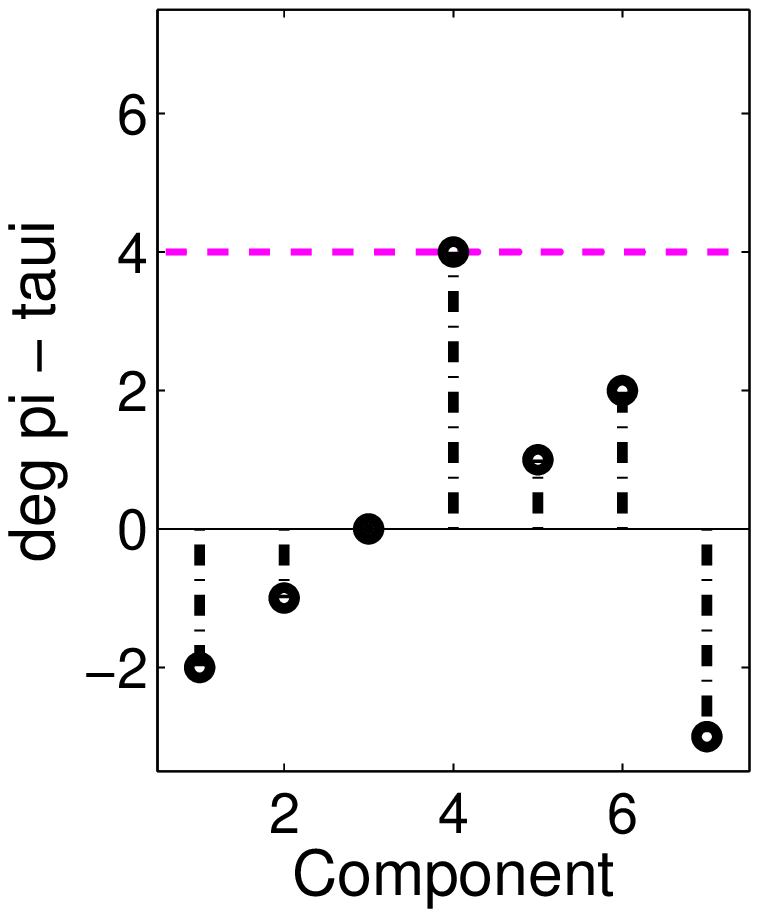} \\
\hspace{10pt}(a) & \hspace{10pt}(b) & \hspace{10pt}(c)
\end{tabular}
\caption{\small\sl A visual representation of the $\tau$-degree: (a) polynomial degrees of each component $p_i(z)$ of a $7\by 1$ polynomial vector $p(z)$; (b) individual components of a sample $\tau$ vector; (c) the degrees of $p_i(z)$ shifted by the $\tau$ values. In (c), the $\tau$-degree is the maximum of the shifted degree values, and is represented by a dashed line. }
\label{fig:taudeg}
\end{figure}

As Fig.~\ref{fig:taudeg} illustrates, the $\tau$-degree is parametric; its value for a fixed $p(z)$ depends on the parameters $\set{\tau_i}$.  This property allows the $\tau$-degree to reflect how closely an element of $\mathcal{S}$ matches $p^*(z)$ in degree structure. To illustrate, we can consider the $\tau$-degree in standard polynomial interpolation.  Suppose we wish to determine the polynomial $u(z)$ of minimal degree that satisfies
\[ u(\omega_k)=f_k\hspace{20pt}\text{for}\hspace{20pt}k=1,\ldots,K.\]
Since there are $K$ interpolation points, $\degr{u}=K-1$.  For a given polynomial $v(z)$ that satisfies the conditions, if $\tau=K-1$ and $\tdg{v}=k$, there are three possibilities: 
\begin{itemize}
\item $k<0$: $v(z)$ will not obey all conditions in general;
\item $k=0$: $v(z)=u(z)$; or 
\item $k>0$: $v(z)$ is not of minimal degree.
\end{itemize}

Next, consider the vector case.  Letting $\tau=[\tau_i]$, where
\begin{equation} 
\tau_i= \begin{cases}
m-1 & i=2,3,4 \\
n-1 & i=1,5,6 \\
0 & i=7 
\end{cases},
\label{eqn:taudef} 
\end{equation}
it follows that $\tdg{p^*}=0$, indicating that $p^*(z)$ has the desired degree structure.  Note that it is not the {\sl specific} values of $\set{\tau_i}$ that matter, but the {\sl relative values}.  For instance, if $\tilde{\tau}=\tau+\bm{1}$, then a $\tilde{\tau}$-degree $\delta=-1$ corresponds to the proper degree structure. 

Since $p^*(z)$ is the only element of $\mathcal{S}$ that satisfies~\eqref{eqn:degcon}, computing the Tikhonov-regularized solution amounts to finding an element of $\mathcal{S}$ with $\tau$-degree $\delta=0$ for $\tau$ as in~\eqref{eqn:taudef}. To find such an element, we make use of a special algebraic property of $\mathcal{S}$: it is {\sl free}, meaning it has a basis. 

Submodule bases play a role similar to their linear-algebraic counterparts, allowing elements of the submodule to be described through expansions. Namely, a set $B(z)=\set{B^{(1)}(z),\ldots,B^{(k)}(z)}$ is a basis for $\mathcal{S}$ if, for every element $p(z)\in\mathcal{S}$, there are {\em unique} polynomials $\alpha_i(z)\in\mathbb{C}[z]$ such that $p$ can be written
\begin{equation}
\label{eqn:basisexp} 
p(z) = \dsum{i=1}{k}{\alpha_i(z)B^{(i)}(z)}.
\end{equation}
The $\set{\alpha_i(z)}$ serve as ``expansion polynomials'' and~\eqref{eqn:basisexp} as a basis expansion.  As is the case for linear subspaces, any element of $\mathcal{S}$ is entirely (and uniquely) described by its $\alpha_i(z)$ for a chosen basis, and the number of bases is infinite.

By Theorem 3.1 of~\cite{Barel:1992}, a basis for $\mathcal{S}$ has exactly seven elements $B^{(j)}(z)\in\mathbb{C}[z]^{7\by 1}$, which may be gathered into the basis matrix polynomial $B(z)=\left[ B^{(1)}(z) \hspace{5pt} \cdots \hspace{5pt}  B^{(7)}(z) \right]$. While we do not explore their theoretical properties, submodule bases are an essential component of our algorithm; we will determine $p^*(z)$ by constructing a basis for $\mathcal{S}$ that has an element with $\tau$-degree $\delta=0$.  Once we build such a basis, the solution $p^*(z)$ will be immediate.

More specifically, to ensure that our basis has an element with $\tau$-degree $\delta=0$, we will construct what is known as a {\sl $\tau$-reduced} basis.  A $\tau$-reduced basis has elements that act ``linearly independent'' relative to the $\tau$-degree, meaning that linear combinations of the basis elements can not decrease the $\tau$ degree.  In other words, if $\delta$ is the maximum $\tau$-degree among the elements of the basis, then there is no linear combination
\[ q(z)=\sum_i \alpha_i(z)B^{(i)}(z)\]
such that $\tdg{q}<\delta$ other than $\alpha_i(z)=0$.

Before describing how we construct a $\tau$-reduced basis, we close the section with two remarks.  First, it is important to note that an arbitrary set of $\tau$-reduced elements of $\mathcal{S}$ is not necessarily a basis; it must also span $\mathcal{S}$.  Second, we construct a $\tau$-reduced basis for $\mathcal{S}$ not out of necessity, but because we can devise an efficient algorithm to do so.  Any efficient method for determining an element of $\mathcal{S}$ with $\tau$-degree $\delta=0$ would suffice.

\subsection{Basis construction}
\label{ssec:construct}

In this section, we detail an algorithm for constructing a $\tau$-reduced basis corresponding to the tangential-interpolation problem. We begin with the following theorem, which allows us to subdivide the process.
\begin{theorem}
\label{thm:subdivide}
Let $\sigma_1,\ldots,\sigma_K$ be interpolation nodes corresponding to the vectors $\phi_1,\ldots,\phi_K\in\mathbb{C}^{1\by 7}$, where the conditions $\set{(\sigma_k,\phi_k)}$ are mutually distinct and $\phi_k\neq\zero^T$ $\forall k$.  For some $1\leq\kappa\leq K$ and $\tau_K\in\mathbb{Z}^{7\by 1}$, let $B_\kappa(z)\in\mathbb{C}[z]^{7\by 7}$ be a $\tau_K$-reduced basis corresponding to the interpolation data $\set{ (\sigma_k,\phi_k)\ :\ k=1,\ldots,\kappa }.$

Denote $\delta_i=\tau_K\mbox{--}\degr{B^{(i)}_\kappa}$ for $i=1,\ldots, 7$, and define 
\[\tau_{\kappa\rightarrow K}=-\begin{bmatrix} \delta_1 & \cdots & \delta_7 \end{bmatrix}^T.\]  
Let $B_{\kappa\rightarrow K}(z)\in\mathbb{C}[z]^{7\by 7}$ be a $\tau_{\kappa\rightarrow K}$-reduced basis matrix corresponding to the interpolation data
\[ \set{ (\sigma_k,\phi_kB_\kappa(\sigma_k))\ :\ k=\kappa+1,\ldots,K }.\]
Then $B_K(z)=B_\kappa(z)B_{\kappa\rightarrow K}(z)$ is a $\tau_K$-reduced basis matrix corresponding to the interpolation data
\[ \set{ (\sigma_k,\phi_k)\ :\ k=1,\ldots,K }.\]
\end{theorem} 
\begin{proof}
See Van Barel and Bultheel~\cite{Barel:1994}, Theorem 3.
\end{proof}

Theorem~\ref{thm:subdivide} provides a method of continually subdividing the interpolation problem into smaller subproblems.  It is also the reason that we compute a $\tau$-reduced basis, as the main result does not hold without the bases being $\tau$-reduced.  However, we still need to solve problems at the finest scale. 

Consider a single interpolation condition, composed of a vector $\phi\in\mathbb{C}^{1\by 7}$ and a node $\sigma$.  Without loss of generality, let $\tau_1$ be the smallest value of $\tau$, and define $\mu_i=-\phi_i/\phi_1$.  Then the polynomial matrix
\begin{equation}
\label{eqn:spbasis}
 B(z) = \begin{bmatrix}
z-\sigma & \mu_2 & \cdots & \mu_7 \\
0 & 1 &  \cdots & 0 \\
\vdots & \vdots & \ddots & \vdots \\
0 & 0  & \cdots & 1
\end{bmatrix}
\end{equation}
satisfies the tangential-interpolation condition
\[ \phi B(\sigma)=\zero^T.\]
Since $\tau_1$ is the minimum $\tau$ value, $B(z)$ is a $\tau$-reduced basis for the submodule $\mathcal{S}_1$ defined by the single interpolation condition $(\sigma,\phi)$ (see~\cite{Barel:1998}).  

We can determine a basis for $\mathcal{S}$ in $\mathcal{O}(n^2)$ operations by processing the conditions serially with the single-point construction of~\eqref{eqn:spbasis} and by invoking Theorem~\ref{thm:subdivide}.  This process corresponds to the ``fast-only'' basis-construction algorithms of~\cite{Heinig:1996} and~\cite{Barel:1998}.  However, the $\mathcal{O}(n^2)$ complexity can be improved by exploiting the structure of the interpolation nodes.  Since the nodes are roots of unity, we can reduce the number of required operations with the recursive ``interleaving'' data-splitting  of Alg.~\ref{alg:divconq}.  In this routine, \textsc{TanInt} is the serial basis-construction function, which for a single point amounts to computing~\eqref{eqn:spbasis}.

\begin{algorithm}
\begin{algorithmic}
\Procedure{$B(z) = $RecTanInt}{$\set{\sigma_k}$, $\set{\phi_k}$}
\State $N=$\textsc{\ length}$(\set{\sigma_k})$
\If {$N=1$}
	\State $B(z)\leftarrow$\textsc{TanInt}$(\set{\sigma_k,\phi_k})$
\Else
	\State $B_L(z)\leftarrow$\textsc{RecTanInt}$\left(\set{\sigma_{2k},\phi_{2k}}\right)$\tabto{3in}\circled{1}
	\For {$k=1,\ldots,N/2$}
		\State $\phi_{2k-1}\leftarrow \phi_{2k-1}B_L(\sigma_{2k-1})$\tabto{3in}\circled{2}
	\EndFor
	\State $B_R(z)\leftarrow$\textsc{RecTanInt}$\left(\set{\sigma_{2k-1},\phi_{2k-1}}\right)$\tabto{3in}\circled{3}
	\State $B(z)\leftarrow B_L(z)B_R(z)$\tabto{3in}\circled{4}
\EndIf
\EndProcedure
\end{algorithmic}
\caption{Recursive tangential-interpolation algorithm for nodes $\set{\sigma_k}$ and vectors $\set{\phi_k}$.}
\label{alg:divconq}
\end{algorithm}

To see how such a strategy is beneficial, we can estimate the total number of operations it requires.  Let $C_n$ be the total cost of calling \textsc{RecTanInt} on a set of $n>1$ points.  Steps \circled{1} and \circled{3} require $C_{n/2}$ operations by definition.  In step \circled{2}, the evaluations $\set{B_L(\sigma_{2k-1})}$ must be computed and multiplied against the $\set{\phi_{2k-1}}$.  Since the $\set{\sigma_k}$ are roots of unity, $B_L(z)$ can be evaluated at the nodes $\set{\sigma_{2k-1}}$ using an FFT of length $n/2$.  Once these values are obtained, there are a total of $p^2n/2$ matrix-vector products (where $p=7$ is the basis size).  Therefore, the number of operations in step \circled{2} is upper-bounded by $c_1 n\log(n)$, where $c_1$ is a constant.  Finally, step \circled{4} requires a number of polynomial multiplications and additions that depends only on the basis size.  Each multiplication can be computed with the FFT, but the FFT size depends on the degrees of the polynomials involved.  By the nature of the basis construction, all of the polynomials in $B_L(z)$ and $B_R(z)$ must have degrees no greater than $n/2$, since they are constructed from $n/2$ conditions.  Therefore, for another constant $c_2$, the number of operations needed to multiply $B_L(z)B_R(z)$ is bounded by $c_2 n\log(n)$.  

Factoring in these costs, we get the recursive cost formula
\[ C_n=2C_{n/2}+cn\log(n),\]
where the constant $c$ is determined by the costs of steps \circled{2} and \circled{4}.  Since this expression is a recurrence, the overall cost of constructing the basis is given by the Master Theorem~\cite{Cormen:2001} as $\mathcal{O}(n\log^2 n)$.  The key to replacing a factor of $n$ from the fast-only basis construction with a factor of $\log^2(n)$ is our ability to evaluate $B_L(\sigma_{2k-1})$ in only $\mathcal{O}(n\log (n))$ operations with the FFT.

To conclude the section, we examine the degree structure of the resulting basis.  Recall from Section~\ref{ssec:frameor} that we can obtain the solution to the Tikhonov-regularization problem if we can construct a $\tau$-reduced basis for $\mathcal{S}$ with a column having $\tau$-degree $\delta=0$.  We now argue that this will indeed be the case. 

For the general Tikhonov problem, there are a total of $3N$ interpolation conditions. In each step of the algorithm, a single condition is used to increase the $\tau$-degree of exactly one column of the current basis -- the column with lowest $\tau$-degree -- by one.  Beginning with no interpolation conditions, and specifying the starting basis as the identity matrix, the $\tau$-degree of the $j^{th}$ column is $-\tau_j$. During each iteration, the column with lowest $\tau$-degree has its degree increased by one (while the remaining $\tau$-degrees are unchanged), and therefore the difference in $\tau$-degree between any two columns of the final basis can be at most one.  Since the values of $\tau$ sum to $3N-6$, the sum of the final $\tau$-degrees is $-(3N-6)+3N=6$.  Since the maximum $\tau$-degree difference between any two columns is one, six of the columns will have $\delta_j=1$ while the seventh has $\delta_j=0$.  Thus, our solution is guaranteed.

\section{A superfast algorithm for Tikhonov regularization}
\label{sec:algorithm}

Our Tikhonov-regularization algorithm can be summarized as follows.  First, we decouple the Gramian matrices $G_T$ and $G_L$ by introducing the artificial variables $\sigma_1$ and $\sigma_2$, turning the original system into the Toeplitz-block system of~\eqref{eqn:ex1}.  We then define the additional artificial variables $\set{\gamma_i}$, which allow us to replace the blocks of this system with partial-circulant matrices using extensions in the form of~\eqref{eqn:circexp}.  The resulting partial-circulant-block system is given in~\eqref{eqn:ex3}.

Once the matrix has been extended, we transform it with a Fourier-block operator to obtain the diagonal-block system of~\eqref{eqn:ex4}.  The coefficients of the diagonal matrices in this expression define tangential-interpolation conditions, and can be calculated through $N$-point FFTs.  Once these conditions are established, we use the divide-and-conquer basis-construction algorithm to build a $\tau$-reduced basis $B^*(z)$, with $\tau$ as in~\eqref{eqn:taudef}.  A single column $j$ of the basis will have $\tau$-degree equal to zero, and the solution $\hat{x}$ is given by
\[ \hat{x}(z) = \dfrac{{B^*_1}^{(j)}(z)}{{B^*_7}^{(j)}}.\]

The basis-construction algorithm can be unstable if not implemented carefully.  As in~\cite{Barel:2001} and~\cite{Barel:2003}, we pivot the interpolation conditions to avoid multiplying by small $\mu_i$ values early in the construction.  Without pivoting, numerical errors may propagate as the algorithm progresses.  

In addition, we take precaution not to process any interpolation conditions for a subproblem that might cause precision errors in {\sl later} subproblems.  If processing an interpolation condition might result in a numerically unstable basis for a subproblem, that condition is skipped and included only after {\sl all} remaining subproblems have been solved. The skipped conditions are marked as ``difficult points,'' and our calculation is more robust if we treat them only at the conclusion of the algorithm. Since the difficult points are processed with the serial basis constructor, the algorithm's efficiency is a function of the number of difficult points we encounter.  For this reason, it is important to avoid generating many difficult points; we shall further discuss this consideration in Section~\ref{sec:implementation}.

\section{Implementation}
\label{sec:implementation}

In this section, we detail two previously unexplored implementation issues with the basis-construction algorithm.  First, we discuss data partitioning and its potential to affect the number of points marked as difficult.  Second, we consider the task of constructing a basis when the number of interpolation conditions is not a power of two, as this has the potential to drastically increase the required overhead.  By addressing these issues, we can maintain a small overhead cost for all matrix sizes and further improve the accuracy of our algorithm.

\subsection{Data partitioning}
\label{ssec:datapart}

In the pseudoinversion algorithm of~\cite{Barel:2003}, the authors separate their basis construction into two stages, each of which uses a separate set of the interpolation conditions.  Extending such an approach to the general Tikhonov problem amounts to computing the final basis as the product
\[ B^*(z)=B_{H}(z)B_{G}(z)B_{F}(z),\]
where the bases $B_H(z)$, $B_G(z)$, and $B_F(z)$ are constructed using the conditions $\set{H_k}$, $\set{G_k}$, and $\set{F_k}$, respectively.  This strategy has the potential to dramatically reduce the number of required calculations, as it allows for smaller basis sizes in the intermediate problems and some data-independent precomputation.  Unfortunately, it also tends to generate many difficult points, making it inefficient in practice.  

To illustrate, consider when the matrix $L$ is square, producing $N=2n$ conditions $\set{H_k}$.  In this scenario, we have the interpolation conditions
\[ \lambda_{3,3}^{(k)}=(-1)^k\hspace{15pt}\text{ and }\hspace{15pt}\lambda_{3,6}^{(k)}=1.\] 
When we subdivide the conditions during the construction of the basis $B_H(z)$, we find that 
\[ \lambda_{3,3}^{(2k)}=\lambda_{3,6}^{(2k)},\hspace{15pt}\text{ and }\hspace{15pt}\lambda_{3,3}^{(2k-1)}=-\lambda_{3,6}^{(2k-1)}.\]
This collinearity is detrimental when we attempt to solve either of the two subproblems.  Specifically, since $\sigma_2(z)$ and $\gamma_3(z)$ have the same degree structure and the same interpolation conditions, they become indistinguishable in the subproblems.

One way to ameliorate this difficulty is to modify the subdivision strategy.  Rather than interleaving subdivisions, we propose the ``paired-interleaving'' of
\begin{eqnarray*}
 H^{(1)} &=& \set{H_1, H_2, H_5, H_6, \cdots } \\
 H^{(2)} &=& \set{H_3, H_4, H_7, H_8, \cdots }.
\end{eqnarray*}
This method will prevent collinearity in conditions corresponding to identity-matrix blocks.  We can still evaluate the basis efficiently at the split nodes, since we have effectively subdivided the data by four while processing two sets at once.  While paired interleaving requires twice as many FFTs to evaluate the basis for a subproblem, these FFTs are half of the size they would be for standard interleaving.   As a result, the number of operations is essentially unchanged.

Paired interleaving can greatly reduce the number of difficult points generated, as it removes most collinearity problems.  Since the difficult points are processed with the fast-only basis constructor, there is a drastic decrease in the number of operations necessary to construct the basis with this strategy.  Additionally, since difficult points generate numerical instability, this approach tends to be less error-prone. 

While paired interleaving helps, a large number of difficult points may still arise if the interpolation conditions for the three block rows are processed in isolation.  Regardless of the subdivision strategy, constructing a basis using only the conditions $\set{H_k}$ amounts to attempting to solve the equation
\[ \sigma_2-Lx=0\]
without any knowledge of $\sigma_2$ or $x$.  In the absence of additional information, there are not enough restrictions on possible solutions to yield a meaningful result.  Accordingly, many difficult points will still be encountered.

Therefore, rather than process the conditions separately, we directly construct the basis from {\sl all sets} of the interpolation conditions.  To retain our ability to evaluate the bases efficiently, our subdivisions use a paired-interleaving split on {\sl each} set of conditions, as illustrated in Fig.~\ref{fig:splitfig}.  While lacking the benefits of reduced basis size and pre-computability, this approach is usually much more efficient for the Tikhonov problems since it generates many fewer difficult points.

\begin{figure}
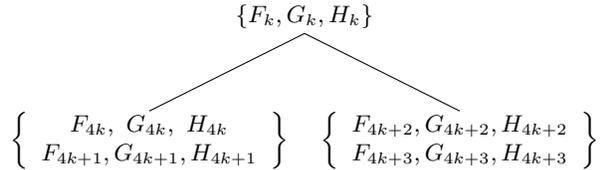

\centering
{\small
\Tree [.$\set{F_k,G_k,H_k}$ $\set{\begin{array}{c}
F_{4k}, \ G_{4k}, \ H_{4k} \\
F_{4k+1}, G_{4k+1}, H_{4k+1}
\end{array}}$ $\set{\begin{array}{c}
F_{4k+2}, G_{4k+2}, H_{4k+2} \\
F_{4k+3}, G_{4k+3}, H_{4k+3}
\end{array}}$ ]
}
\caption{\small\sl Subdivision process during an iteration of the basis construction.  Each set of conditions is subdivided into two components using the paired-interleaving strategy.  Corresponding components for each set of conditions are processed together.}
\label{fig:splitfig}
\end{figure}

\subsection{Data sizes}
\label{ssec:datasize}

We have so far assumed that the algorithm will subdivide the conditions until arriving at a single interpolation point.  As noted in~\cite{Barel:2001}, this is a poor strategy in practice.  At some level, further subdivision results in a higher overhead cost than serial construction.  Hence, our implementation calls the fast-only solver when the number of interpolation conditions for a given subproblem satisfies $K\leq N_{lim}$ for some specified $N_{lim}$.  The optimal value of $N_{lim}$ is machine-dependent; we have empirically found that it tends to lie in the range of 256--512 for our machines. 

We have also implicitly assumed that the total number of conditions we wish to process is a power of two.  If this is indeed the case, the subdivided data points $\set{\sigma_{2k-1}}$ are roots of $-1$, and we can compute $B_L(\sigma_{2k-1})$ using length-$N/2$ FFTs. When $N$ is not a power of two, however, we require longer FFTs.

For instance, consider the case when $N$ is prime; there is no cancellation in the complex exponentials of $\set{\sigma_{2k-1}}$. As a result, we require a full length-$N$ FFT to evaluate $B_L(z)$ at these nodes regardless of how many points there are in the subdivision.  This increased overhead can cause the computation time to vary drastically across problem sizes.

To counteract this effect, we return to the circulant extensions of~\eqref{eqn:circexp}.  For convenience, we developed our basis construction algorithm using $1$-circulant extensions.  However, we may instead choose to perform $k$-circulant extensions, where we choose $k\geq 0$ to allow us to compute shorter FFTs when evaluating $B_L(z)$.  We have found empirically that setting the arbitrary entries in the extensions to zero when $k\geq 1$ can lead to very ill-conditioned problems.  To avoid this, we choose them to be similar in value to the known matrix coefficients.  While we have no theoretical results to justify such a tactic, it appears to be more stable for the problems we have tested.

The question of how to choose an optimal $k$ remains. One immediate option, following the lead of~\cite{Barel:2001}, is to choose $k$ such that $N$ is a power of two; we refer to this as the ``power-of-two'' method. The problem with such an approach is that it will cause the algorithm's complexity function to become highly quantized.  In particular, the number of operations required for $2^p$ points will be significantly lower than it will be for $2^p+1$.

Instead, we propose an alternative which may also be applied to the algorithms of~\cite{Barel:2001} and~\cite{Barel:2003}.  Since we would like each subdivision to produce cancellation in the complex exponentials of the roots of unity, our requirements are:
\begin{enumerate}
\item $N=2^p M$, where $M$ is an integer;
\item $3M \leq N_{lim}$; and
\item $N\geq \tilde{N}$, where $\tilde{N}$ is the number of conditions we would have in each stage with $0$-circulant extensions.
\end{enumerate}
The first condition guarantees cancellation in the complex exponential for the first $p$ subdivisions.  The second condition guarantees that after $p$ subdivisions, the number of remaining points is sufficiently small that the fast-only basis constructor is called.  The final condition assures that the extensions will indeed produce partial-circulant blocks.  These conditions are devised for an interleaving data split; since we are using a paired-interleaving split in practice, the right side of the inequality in the second condition is replaced by $N_{lim}/2$ to ensure that we can evaluate the nodes $\set{\sigma_{4k+i}}$ efficiently.

Alg.~\ref{alg:choosek} provides a simple method for choosing the extension size $k$ to satisfy these criteria.  The algorithm is formulated for interleaving data splitting, and can be easily modified for paired-interleaving splitting by changing the {\sc while} loop condition.  It is also more efficient to force $M$ to be even; this ensures that none of the smallest FFT sizes are prime numbers, reducing the overhead cost of the FFT calls.

\begin{algorithm}
\begin{algorithmic}
\Procedure{$k$ = OptExtend}{$\tilde{N}$,$N_{lim}$}
\State $M\leftarrow \tilde{N}$
\State $M_{tot} \leftarrow 3M$
\State $p \leftarrow 0$
\While {$M_{tot}> N_{lim}$}
\State $p \leftarrow p+1$
\State $M \leftarrow \left\lceil M/2\right\rceil$
\State $M_{tot} \leftarrow 3M$
\EndWhile 
\State $k \leftarrow 2^pM - \tilde{N}$
\EndProcedure
\end{algorithmic}
\caption{Routine for calculating the optimal circulant extension size for minimum number of conditions $\tilde{N}$ and fast-only level $N_{lim}$}
\label{alg:choosek}
\end{algorithm}

Fig.~\ref{fig:ncond} plots the number of interpolation conditions processed for the power-of-two method {\sl versus} our proposed method.  For all data points in the curves, sufficient cancellation occurs to halve the required FFT length each time the nodes are subdivided.  Since the number of conditions is lower, the complexity for our proposed method should increase more smoothly with the problem size.  However, the actual gain in efficiency may depend on the FFT implementation.

\begin{figure}[htb]
\centering
\psfrag{Minimal no. conditions per stage}{\raisebox{-4pt}{\fontsize{10}{11.2}\selectfont Minimal no. conditions per stage}}
\psfrag{Conditions per stage}{\raisebox{4pt}{\fontsize{10}{11.2}\selectfont Conditions per stage}}
\psfrag{Power-of-Two Method}{\fontsize{8}{9.6}\selectfont Power-of-Two Method}
\psfrag{Proposed Method}{\fontsize{8}{9.6}\selectfont  Proposed Method}
\psfrag{16384}{\fontsize{8}{9.6}\selectfont 16384}
\psfrag{8192}{\fontsize{8}{9.6}\selectfont 8192}
\psfrag{4096}{\fontsize{8}{9.6}\selectfont 4096}
\psfrag{2048}{\fontsize{8}{9.6}\selectfont 2048}
\includegraphics[width=4.0in]{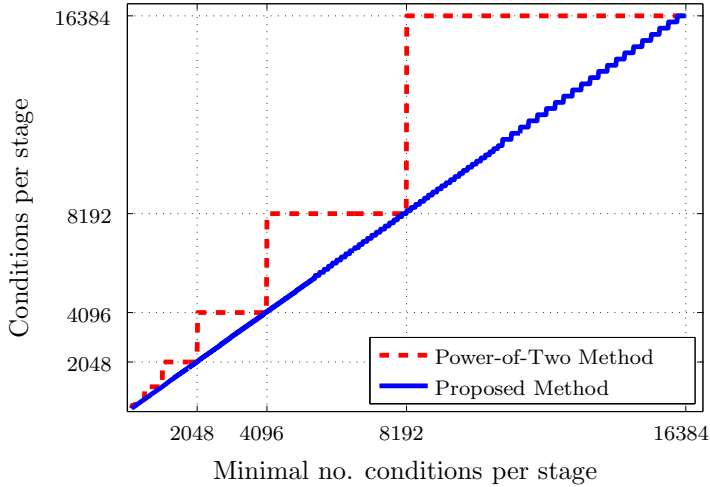}
\vspace{5pt}
\caption{\small\sl Number of conditions processed {\sl versus} the minimal possible number of conditions for the power-of-two method and our proposed method of circulant extension, with $N_{lim}=256$.  The ``quantization'' in the number of conditions processed with our proposed method is significantly milder than that of the power-of-two method.}
\label{fig:ncond}
\end{figure}

\section{Numerical simulations}
\label{sec:simulations}

We now describe several experiments that demonstrate our algorithm's utility.  First, we present the results of example problems with randomly generated matrices to illustrate that the algorithm's computational complexity increases as $\mathcal{O}(n\log^2 n)$.   Next, we compare the tangential interpolator to an iterative Tikhonov solver based on the Conjugate Gradient (CG) method by expressing its execution time in terms of an equivalent number of CG iterations.  Finally, we compute the regularized recovery of a low-frequency signal from its non-uniform Fourier samples.

Our algorithm is coded in C++ with {\sc MEX}-function interfaces to {\sc MATLAB}.  All experiments were run on a 3.16 GHz Intel Core 2 Duo machine with 3.0 GB of RAM under the Ubuntu 12.04 LTS operating system and {\sc MATLAB R2012a}.\footnote{The code for these experiments can be obtained from \url{http://users.ece.gatech.edu/~cturnes3/software.html}.} For each experiment, we used paired-interleaving data splitting and the extension strategy of Section~\ref{ssec:datasize}. 

\subsection{Computational complexity}
\label{ssec:comp}

To confirm the asymptotic cost of our algorithm, we solved a large number of Tikhonov problems across a range of matrix sizes.  Since an exact solution to the Tikhonov-regularized system $Tx=b$ will not return a vector identical to $x$, we instead used our algorithm to solve systems of the form
\[ (G_T+G_L)x=y\]
for known input vectors $x$.  These systems are effectively equivalent to Tikhonov-regularization problems, as we could replace $y$ with $T^HTx$.  However, by comparing our solutions
\[ \hat{x}=(G_T+G_L)^{-1}y\]
with the known input vectors, we can analyze our algorithm's accuracy in applying the matrix $(G_T+G_L)^{-1}$ (see Section~\ref{ssec:acc}).

In each experiment, the system matrix $T$ and Tikhonov matrix $L$ were $n\by n$ Toeplitz matrices whose coefficients were drawn from a (complex) standard normal distribution, for a total of $4n-2$ free parameters per experiment. The value of $n$ was varied in even logarithmic steps between $2^9$ and $2^{15}$, with the resulting average execution times shown in Fig.~\ref{fig:spectikcomp}(a).  In the figure, the execution time is plotted as a function of the number of free parameters rather than the matrix side length.

To verify the algorithm's complexity, we compared the observed execution times to the theoretical behavior. Based on the algorithm's operation count, if the execution time is primarily a function of the number of operations it should be characterized by a function of the form
\[ E(n)=c_1n\log^2 n+c_2n\log n+\varepsilon(n)+\text{ov}(n),\]
where $c_1$ and $c_2$ are constants, $\varepsilon(n)$ consists of lower-order computations, and $\text{ov}(n)$ reflects differences in overhead cost for various problem sizes.  For most problems, the $\mathcal{O}(n\log^2 n)$ calculations dominate the computational cost, but the contribution of  $\mathcal{O}(n\log n)$ calculations is non-negligible for the problem sizes we consider.  We can therefore make the approximation 
\[ E(n)\approx c_1n\log^2 n+c_2n\log n.\]
Using this model, we computed a least-squares fit of the constants $c_1$ and $c_2$ and superimposed the estimated computational cost $E(n)$ on the observed calculation times in Fig.~\ref{fig:spectikcomp}(a).

To compare how the execution time of the tangential interpolator is reduced for $\ell_2$-norm penalization and the Toeplitz-Gramian problem, we repeated the experiments for these cases.  For $\ell_2$-norm penalization, we fixed the regularization parameter $\beta$ for each matrix size. Similarly, we fixed the main diagonal $a_0$ of $G_T$ for each of the Toeplitz-Gramian experiments.  Neither of these choices affect the execution time; we have specified these values only to better control the matrix conditioning, allowing for more meaningful comparisons of accuracy (as will be explained in the next section).  As a result of these choices, the number of free parameters for these problems are $2n-1$ and $3n-2$, respectively.  In Fig.~\ref{fig:spectikcomp}(b) and (c), the average execution times for these problems are plotted as functions of the number of free parameters, and are again compared to least-squares fits of the underlying cost function.

\begin{figure*}[!t]
\centering
\begin{tabular}{ccc}
\psfrag{Execution times}{\fontsize{6}{7.2}\selectfont Execution time}
\psfrag{Complexity estimate E(n)}{\fontsize{6}{7.2}\selectfont Complexity E(n)}
\psfrag{General Problem}{\fontsize{10}{12}\selectfont \hspace{-8pt}{\bf General Problem}}
\psfrag{Total No. of Parameters in Problem}{\raisebox{-3pt}{\fontsize{8}{9.6}\selectfont \hspace{-7pt}Parameters in Problem}}
\psfrag{Execution time (s)}{\raisebox{3pt}{\fontsize{8}{9.6}\selectfont \hspace{-5pt}Execution time (s)}}
\includegraphics[width=2in]{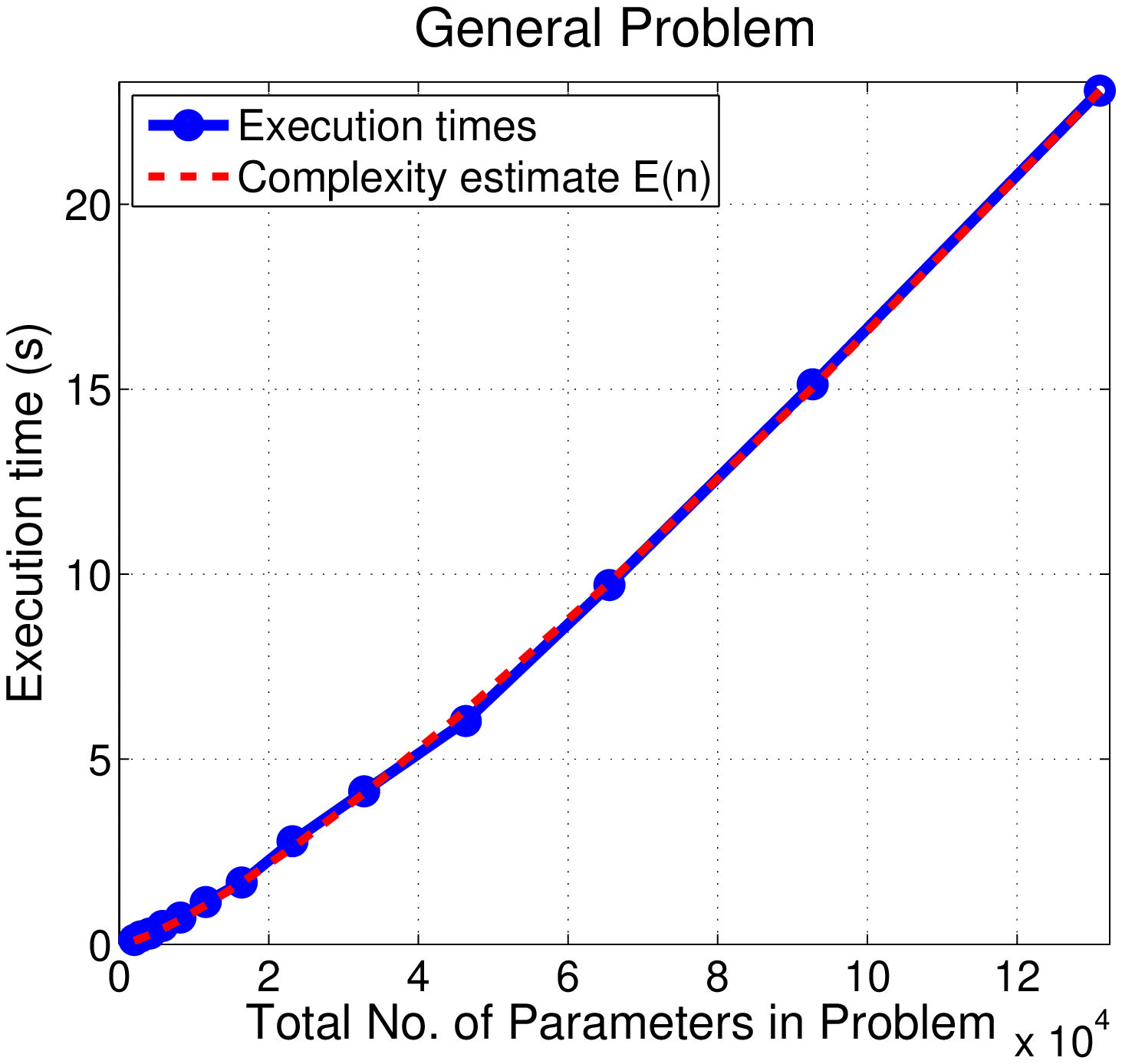} & 
\psfrag{Execution times}{\fontsize{6}{7.2}\selectfont Execution time}
\psfrag{Complexity estimate E(n)}{\fontsize{6}{7.2}\selectfont Complexity E(n)}
\psfrag{Euclidean Norm Penalization}{\fontsize{10}{12}\selectfont \hspace{-2pt}{\bf $\ell_2$-Norm Penalization}}
\psfrag{Total No. of Parameters in Problem}{\raisebox{-3pt}{\fontsize{8}{9.6}\selectfont \hspace{-7pt}Parameters in Problem}}
\psfrag{Execution time (s)}{\raisebox{3pt}{\fontsize{8}{9.6}\selectfont \hspace{-5pt}Execution time (s)}}
\includegraphics[width=2in]{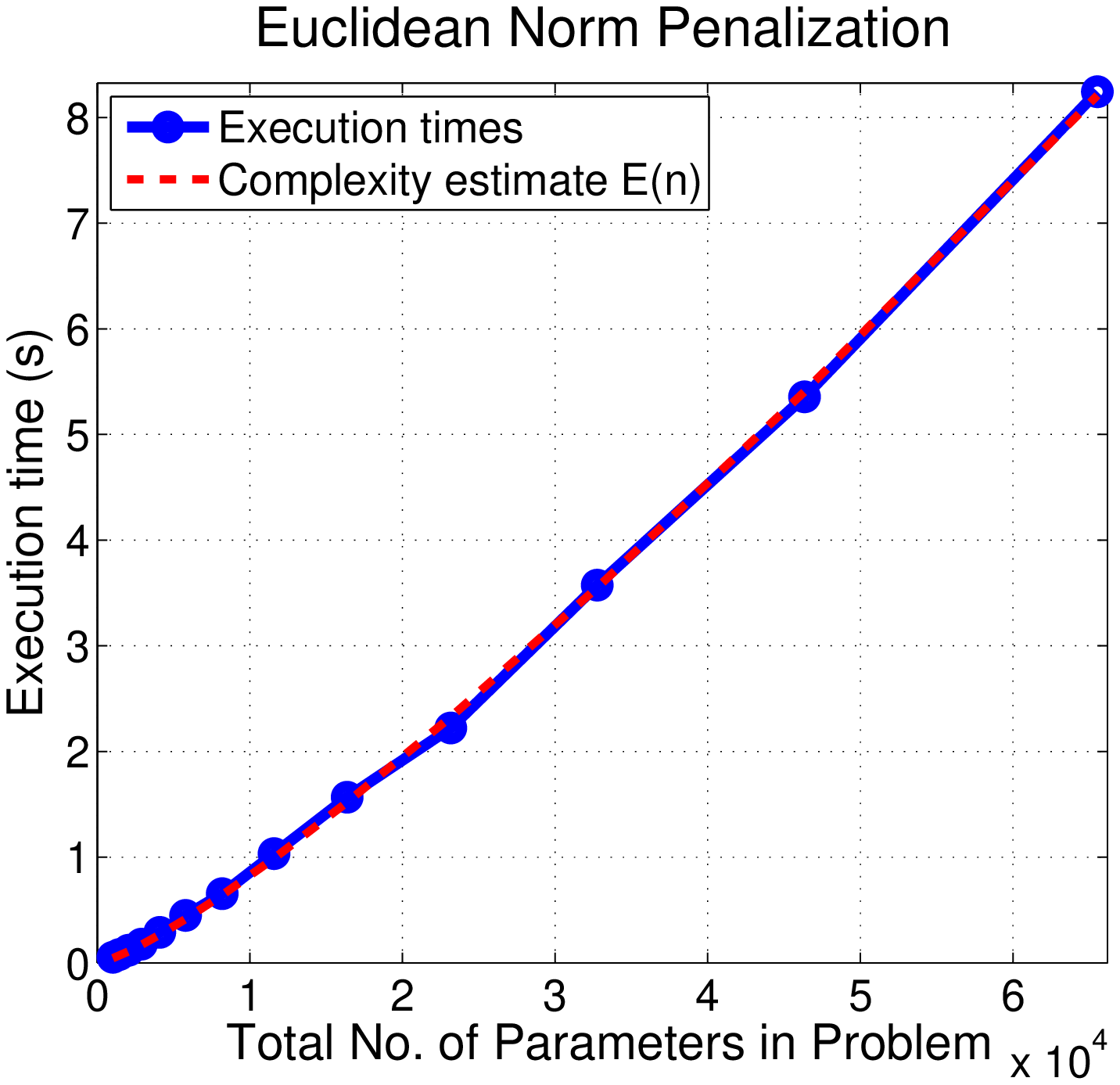} & 
\psfrag{Execution times}{\fontsize{6}{7.2}\selectfont Execution time}
\psfrag{Complexity estimate E(n)}{\fontsize{6}{7.2}\selectfont Complexity E(n)}
\psfrag{Toeplitz Gramian}{\fontsize{10}{12}\selectfont \hspace{-8pt}{\bf Toeplitz Gramian}}
\psfrag{Total No. of Parameters in Problem}{\raisebox{-3pt}{\fontsize{8}{9.6}\selectfont \hspace{-7pt}Parameters in Problem}}
\psfrag{Execution time (s)}{\raisebox{3pt}{\fontsize{8}{9.6}\selectfont \hspace{-5pt}Execution time (s)}}
\includegraphics[width=2in]{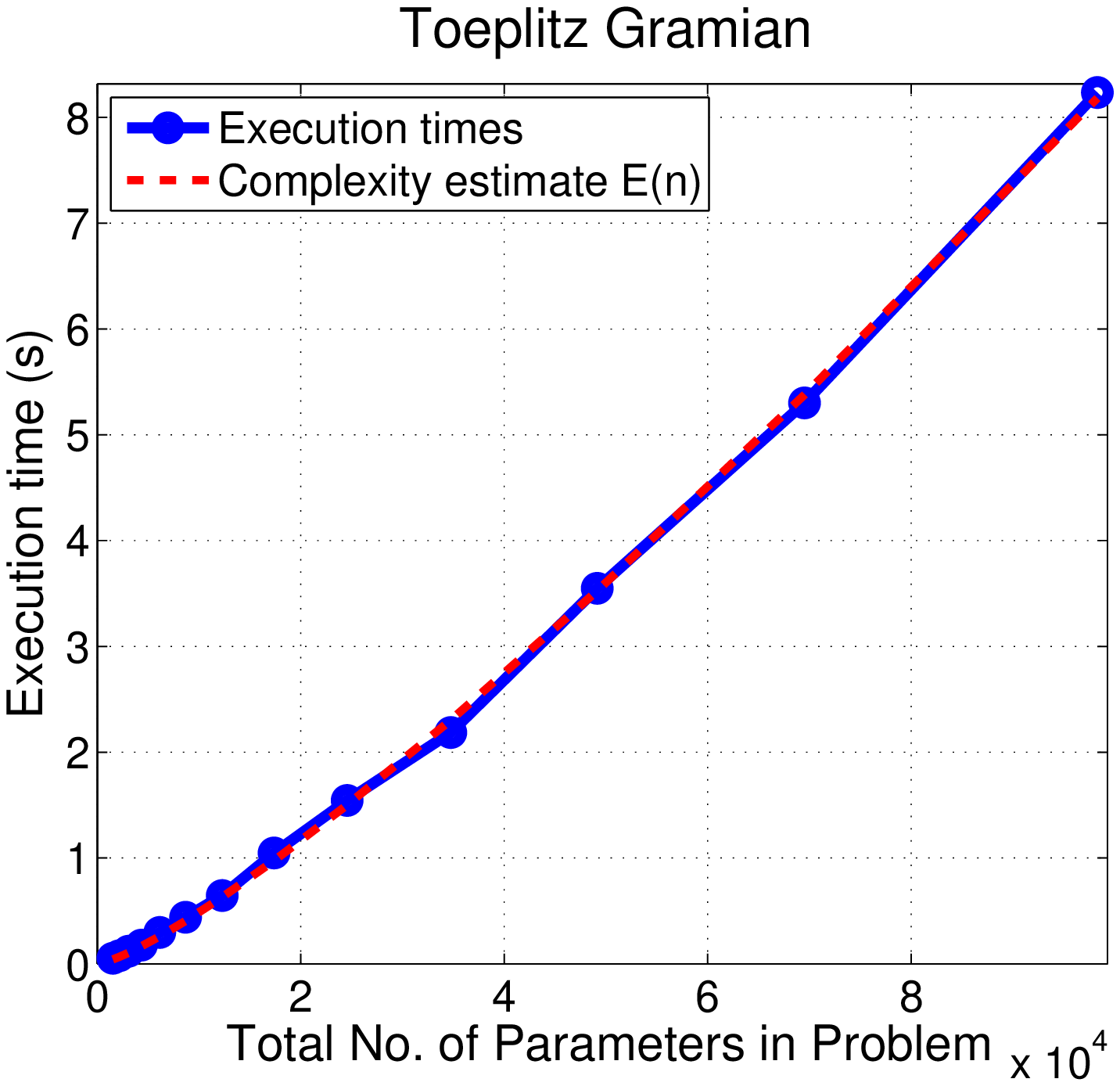} \vspace{10pt}\\
(a) & (b) & (c)
\end{tabular}
\caption{\small\sl Execution time vs. total number of free parameters for (a) square Toeplitz matrices $T$ and $L$ with coefficients drawn from a complex standard normal distribution; (b) square Toeplitz matrices $T$ with coefficients drawn from a complex standard normal distribution and $\ell_2$-norm penalization; and (c) Toeplitz Gramian matrices $G_T$ and Tikhonov matrices $L$ with coefficients drawn from a complex standard normal distribution.  For each data point, the execution time was averaged over 1000 trials.  Also plotted in the figures are the least-squares estimates of the underlying complexity functions, which closely match the observed data.}
\label{fig:spectikcomp}
\end{figure*}

The execution-time curves in Fig.~\ref{fig:spectikcomp} suggest that our algorithm achieves the predicted asymptotic complexity.  Moreover, these experiments give insight into the algorithm's overhead cost for the various problems.  In each plot, the $i^{th}$ data point corresponds to systems of the same size.  The $\ell_2$-norm-penalization problem and the Toeplitz-Gramian problem each require less overhead, as their computation times are significantly smaller than the general problem for systems of the same size.  This result is not surprising; the two special cases involve fewer interpolation conditions and smaller basis sizes than the general problem.

However, we can also consider the cost as a function of free parameters.  In this sense, there is minimal difference in the interpolator's performance for the general problem and for $\ell_2$-norm penalization, but it is more efficient for the Toeplitz-Gramian problem.  While the general problem requires more operations than $\ell_2$-norm penalization, it also contains roughly twice as many parameters.  When considering cost as a function of free parameters, these two effects nearly ``cancel out,'' and the execution time -- as a function of the number of free parameters -- is similar for the two problems.  By contrast, when the Gramian $G_T$ is Toeplitz, the amount of calculation is the same as for a $\ell_2$-norm-penalization problem of the same size, while the number of free parameters increases to $3n-2$.  Therefore, the execution time as a function of the number of parameters drops by a factor of $2/3$ for the Toeplitz-Gramian problem.

\subsection{Accuracy of results}
\label{ssec:acc}

In addition to confirming our algorithm's complexity, we used the experiments of Section~\ref{ssec:comp} to verify that it is able to apply the inverse of the matrix $(G_T+G_L)$ with acceptable accuracy.  For each experiment, we recorded the maximum error between the source vector $x$ and the recovered vector $\hat{x}$.  We then took the maximum error across all trials for each type of problem and matrix size.  

While the tangential interpolator's execution time is seemingly independent of the matrix conditioning, its accuracy is not.  As is to be expected, there are larger errors that propagate throughout the basis construction as the matrix $(G_T+G_L)$ becomes more poorly conditioned, resulting in less exact solutions.  To address this, we fixed certain parameters in our $\ell_2$-norm-penalization and Toeplitz-Gramian experiments to ensure that the matrices generated for each of the three problems had similar condition numbers when all else was equal.  The resulting data better reflects the interpolator's accuracy across the three problems when conditions are effectively the same.

We have empirically found that when the matrix entries are drawn from a standard normal distribution, the sum of the two Gramians $(G_T+G_L)$ tends to be fairly well conditioned for the general problem, even as the matrix size increases.  However, for the $\ell_2$-norm-penalization problem, the conditioning of the single Gramian $G_T$ worsens with increasing size.  To compensate, we set the regularization parameter $|\beta|^2=\sqrt{N}$.  Similarly, the conditioning of the matrix $(G_T+L^HL)$ in the Toeplitz-Gramian problem appears to be largely a function of the diagonal dominance of $G_T$.  We therefore set the main diagonal entry of $G_T$ to be $a_0=10\sqrt{N}$.  The values of $\beta$ and $a_0$ were both found through trial-and-error.

The maximum-error results are listed in Table~\ref{tab:errors}, and suggest that our algorithm applies the inverse with acceptable accuracy for each problem. These results can be further improved with iterative refinement as described in~\cite{Barel:2001} (though this is beyond the scope of this work).  However, the numbers in Table~\ref{tab:errors} do not include any such adjustments, and report the errors after a single pass of the inversion process.  All stabilization measures are intrinsic to the basis-construction algorithm and are factored into the computation times.

\begin{table}[!t]
\renewcommand{\arraystretch}{1.3}
\caption{\small\sl Maximum error between the source vector and the solution returned from the inversion program for the experiments of Section~\ref{ssec:comp}.  For each problem type and matrix size, the maximum errors in the recovered vectors were taken across all 1000 trials. ``Matrix size'' refers to the side length of the matrices for the experiments.}
\centering 
\rowcolors{2}{lightgray}{white}
\begin{tabular}{|C{0.75in}|| C{1in} | C{1in} | C{1in} |} \hline
\textbf{Matrix Size} & \textbf{General Problem} & \textbf{$\ell_2$-Norm Penalization} & \textbf{Toeplitz Gramian} \\ \hline
512  	& 	\sci{9.86}{12}		&	\sci{1.55}{11}	&	\sci{3.73}{12}	\\ 
1024	&	\sci{2.68}{11}		&	\sci{4.38}{11}	&	\sci{1.38}{11}	\\ 
2048	&	\sci{7.56}{11}		&	\sci{1.34}{10}	&	\sci{3.62}{11}	\\ 
4096	&	\sci{1.77}{10}		&	\sci{3.79}{10}	&	\sci{1.19}{10}	\\ 
8192	&	\sci{4.46}{10}		&	\sci{1.13}{9\ }	&	\sci{3.18}{10}	\\ 
16384	&	\sci{1.19}{9\ }		&	\sci{3.39}{9\ }	&	\sci{1.05}{9\ }	\\ 
32768   &	\sci{2.88}{9\ }		&   \sci{1.07}{8\ } &	\sci{3.04}{9\ }	\\ \hline
\end{tabular}
\label{tab:errors}
\end{table}

\subsection{Equivalence in Conjugate Gradient Iterations}
\label{ssec:cg}

The results of Sections~\ref{ssec:comp} and~\ref{ssec:acc} give an absolute measure of the performance of the tangential interpolator.  To gain a sense of perspective, we now translate this performance into a comparison with CG. Since CG is one of the most celebrated iterative methods for solving least-squares problems, it serves as a reasonable benchmark.

Unfortunately, a direct comparison with CG is difficult, as it is an iterative method while the tangential interpolator is not.  Moreover, the convergence speed of CG is dependent on the conditioning of the matrix (among other factors), resulting in variable solution times across different problems of the same size.  By contrast, our algorithm is nearly static in computation time for a given matrix size, as the complexity is primarily dependent on the number of interpolation conditions.  However, we can compare the relative efficiency of the two algorithms by determining how many iterations of CG can be performed in the same amount of time that our superfast solver requires. 

We repeated the experiments of Section~\ref{ssec:comp} for matrices with side-lengths of $n=2^k$, $k=9,\ldots,15$.  For each trial, the time required to compute the explicit inverse of $(G_T+G_L)$ using tangential interpolation was recorded.  We then passed the matrix parameters to an implementation of CG to obtain the solution to the Tikhonov problem, stopping the program when the total elapsed time surpassed the direct-inversion time.  When CG terminated, we recorded the number of {\em completed} iterations (discarding the results of any partial iterations), and averaged it across all trials. The resulting equivalent iteration counts are given in Table~\ref{tab:fftexec}.

To arrive at a fair comparison, we implemented routines that allowed the CG solver to apply the matrix $(G_T+G_L)$ with minimal complexity.  In each case, since either the matrices $T$ and $L$ or their Gramians are Toeplitz, the individual matrices can be applied with FFTs of length $3n-2$.  However, the minimum number of FFTs is different for each problem:
\begin{itemize}
\item General problem: nine -- five to compute $Tx$ and $Lx$ and four to compute $T^H(Tx)$ and $L^H(Lx)$;
\item $\ell_2$-norm penalization: five to compute $T^H(Tx)$;
\item Toeplitz-Gramian problem: seven -- five to compute $G_Tx$ and $Lx$ and two to compute $L^H(Lx)$.
\end{itemize}
In these tallies we have used the fact that the FFT of the generating vector for the matrix $T^H$ can be obtained in $\mathcal{O}(n)$ operations from the FFT of the generating vector for the matrix $T$ (and similarly for $L^H$).

\begin{table}[!t]
\renewcommand{\arraystretch}{1.3}
\caption{\small\sl Number of CG iterations corresponding to the tangential-interpolation Tikhonov solver for the three different Tikhonov-regularization problems.  The table values were calculated by averaging the equivalent number of iterations in each scenario over 1000 trials. ``Matrix size'' refers to the matrix side length for the experiments.}
\centering 
\rowcolors{2}{lightgray}{white}
\begin{tabular}{|C{0.75in}|| C{1in} | C{1in} | C{1in} |} \hline
\textbf{Matrix Size} & \textbf{General Problem} & \textbf{$\ell_2$-Norm Penalization} & \textbf{Toeplitz Gramian} \\ \hline
512   	& 	70.0	&	54.2		&  41.9 \\
1024 	&   74.6	&	55.7		&  41.2 \\
2048 	&	84.3	&	60.0		&  43.8 \\
4096 	&	118.6	&	82.0		&  56.1 \\
8192 	&	126.9	&	84.7		&  58.3 \\
16384 	&	169.7	&	104.4		&  77.4 \\ 
32768 	&	240.2	& 	123.8		&  86.6 \\ \hline
\end{tabular}
\label{tab:fftexec}
\end{table}

As indicated in Table~\ref{tab:fftexec}, the equivalent number of CG iterations increases with the matrix side-length.  This is to be expected; the CG iterations use $\mathcal{O}(n\log n)$ operations while our algorithm contains procedures requiring $\mathcal{O}(n\log^2 n)$ operations.  Therefore, it is unsurprising that as $n$ increases, more CG iterations can be run in the same amount of time it takes for the tangential interpolator to build a solution.

In addition, we may compare the accuracy of the two algorithms for the same execution time.  For each experiment, we recorded the maximum error in the solutions returned by the two algorithms.  We then took the maximum of these errors across all trials, with the results shown in Table~\ref{tab:cgerrors}.

\begin{table*}[!t]
\centering 
\caption{\small\sl Maximum errors between the source vector and the returned solutions of both the CG method and the tangential-interpolation algorithm for the three Tikhonov problem types.  For each problem type and matrix size, the maximum errors in the recovered vectors were computed across all trials. In all trials, the CG method was terminated after surpassing the time required for the inversion program to return a solution to the same problem. ``Matrix size'' refers to the side length of the matrices for the experiments.}
\rowcolors{5}{white}{lightgray}
\begin{tabular}{|C{0.7in}||C{0.625in}|C{0.625in}||C{0.75in}|C{0.75in}||C{0.625in}|C{0.625in}||} \hline
\multicolumn{1}{|c||}{\multirow{3}{0.7in}{\centering\bf Matrix Size}} & 
\multicolumn{2}{c||}{\multirow{2}{1.25in}{\centering\bf General Problem}} & 
\multicolumn{2}{c||}{\multirow{2}{1.5in}{\centering\bf \vbox{$\ell_2$-Norm Penalization}}} &
\multicolumn{2}{c||}{\multirow{2}{1.25in}{\centering\bf Toeplitz Gramian}} \\ 
\multicolumn{1}{|c||}{} & \multicolumn{2}{c||}{} & \multicolumn{2}{c||}{} & \multicolumn{2}{c||}{} \\\cline{2-7}
 & CG & Direct & CG & Direct & CG & Direct\\\hline 
512   	& \sci{1.64}{7\ } & \sci{8.66}{12}  & \sci{1.02}{3\ } & \sci{1.45}{11}  & \sci{1.21}{4\ } & \sci{3.84}{12} \\
1024 	& \sci{1.61}{7\ } & \sci{2.52}{11}  & \sci{5.26}{3\ } & \sci{4.45}{11}  & \sci{4.18}{3\ } & \sci{1.30}{11} \\ 
2048 	& \sci{6.03}{8\ } & \sci{7.49}{11}  & \sci{1.21}{2\ } & \sci{1.28}{10}  & \sci{1.22}{3\ } & \sci{4.02}{11} \\
4096 	& \sci{1.14}{9\ } & \sci{1.70}{10}  & \sci{1.22}{2\ } & \sci{4.00}{10}  & \sci{1.48}{4\ } & \sci{1.29}{10}	\\
8192 	& \sci{2.02}{10}  & \sci{4.32}{10}  & \sci{1.62}{2\ } & \sci{1.14}{9\ } & \sci{3.51}{3\ } & \sci{3.53}{10}	\\
16384 	& \sci{2.28}{10}  & \sci{1.08}{9\ } & \sci{5.85}{2\ } & \sci{3.51}{9\ } & \sci{2.20}{4\ } & \sci{1.01}{9\ } \\
32768 	& \sci{1.09}{14}  & \sci{2.75}{9\ } & \sci{3.82}{2\ } & \sci{1.10}{8\ } & \sci{1.01}{2\ } & \sci{3.23}{9\ } \\ \hline
\end{tabular}
\label{tab:cgerrors}
\end{table*}

Table~\ref{tab:cgerrors} reflects the potential performance gains that can be realized with our algorithm.  The tangential interpolator is only markedly outperformed by the CG solver when the matrices of the general problem become very large.  For $\ell_2$-norm penalization and the Toeplitz-Gramian problem, CG requires much more time to achieve a comparable level of accuracy than the tangential interpolator even when the condition numbers are kept reasonable.

\subsection{Non-Uniform Fourier Experiments}
\label{ssec:nufft}

To demonstrate a practical use of our algorithm, we applied it to the task of reconstructing a signal from non-uniform samples of its spectrum.  This is a common problem in signal processing, and one that is particularly relevant (albeit in a multidimensional variant) for magnetic resonance imaging (MRI).  For one-dimensional signals, the spectrum of a discrete signal $x=[x_s]$ may be sampled at an arbitrary frequency $f_0 \in[-1/2, 1/2)$ by evaluating the sum
\[ X(f_0)=\dsum{s}{}{x_s\ema^{-\jma 2\pi f_0s}}.\]
Collectively, a set of $K$ spectral samples at frequencies $\set{f_k}$ may be obtained by evaluating the matrix-vector product $X = Ax$, where $A$ is a Fourier-like matrix with entries
\[ A_{ks} = \ema^{-\jma 2\pi f_k s}.\]

When $K=N$ and the $\set{f_k}$ are uniformly spaced in $[-1/2,1/2)$, the matrix $A$ is a Fourier matrix. Accordingly, the signal $x$ may be recovered from its spectral samples by
\[ A^HX=A^HAx=x.\]
However, when the frequencies are not uniformly spaced, the matrix $A$ is often severely ill-conditioned, and a regularized solution to the system $X=Ax$ is required. The Tikhonov regularization for the problem is typically formulated as
\[ \hat{x} = (A^HWA+L^HL)^{-1}A^HWX,\]
where $W$ is a diagonal weighting matrix that compensates for the sampling density in the Fourier domain and $L$ is the regularizer.  For MRI reconstruction problems, a Voronoi-cell weighting is typically used to compute the sampling density~\cite{Rasche:1999}, yielding the entries of $W$.  

Straightforward calculations show that the ``weighted Gramian'' $A^HWA$ is structured, with entries
\[ (G_A)_{rs}=(A^HWA)_{rs}=\dsum{k=1}{K}{w_{kk}\ema^{\jma 2\pi f_k(r-s)}}.\]
Since the entries of $G_A$ depend only on the index difference $(r-s)$, $G_A$ is Toeplitz.  If the Tikhonov matrix $L$ is also Toeplitz, this amounts to the Toeplitz-Gramian problem, and our algorithm may be used to determine the closed-form solution to the Tikhonov-regularization problem.

To demonstrate, we generated length-$4096$ input signals consisting of linear combinations of sinusoids of three randomly-chosen frequencies in the digital frequency range $[0, 0.02]$.  For each input signal, we acquired $4096$ spectral samples at random frequencies chosen from a triangular distribution over $[-1/2,1/2)$.  As the number of Fourier samples is equal to the length of the input signals, the corresponding Gramian matrices $G_A$ were severely rank-deficient.  

Since the input signals contained only low-frequency sinusoids, a scaled second-order difference matrix 
\[ L_{r-s} = \begin{cases}
-\text{\sci{1}{4}} & |r-s| = 1 \\
\text{\sci{2}{4}} & r-s = 0 \\
0 & \text{else}
\end{cases}\]
serves as a mild but effective regularizer, penalizing solutions with high-frequency content.  

A typical reconstruction is shown in Fig.~\ref{fig:nufft}.  In this example, the Gramian $G_A$ had a numerical rank of $3050$ despite being $4096\times 4096$, and the condition number of the matrix $(G_A+L^HL)$ was approximately \scip{9.7}{6}.  Despite this somewhat poor conditioning, we were able to acquire a reasonable reconstruction through direct inversion in less than $0.7$ seconds.

\begin{figure}[!t]
\centering
\psfrag{Segment of Signal}{\raisebox{-3pt}{\fontsize{9}{10.8}\selectfont \hspace{-5pt}Segment of Signal}}
\psfrag{Nonuniform Fourier Reconstruction from 4096 Samples}{\fontsize{10}{12}\selectfont \hspace{0pt}{\bf Nonuniform Fourier Reconstruction (4096 Samples)}}
\psfrag{Original Signal}{\fontsize{8}{9.6}\selectfont Original Signal}
\psfrag{Reconstructed Signal}{\fontsize{8}{9.6}\selectfont Reconstructed Signal}
\includegraphics[width=4in]{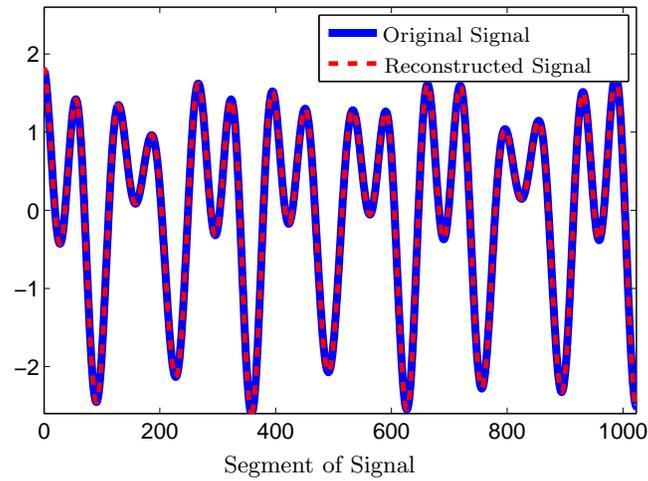}
\caption{\small\sl A segment of the reconstruction of a low-frequency input signal using $4096$ non-uniformly spaced Fourier samples {\sl via} superfast Tikhonov regularization, where the Tikhonov matrix was a scaled second-order difference matrix.  Despite the poor conditioning of the matrix, a reasonable reconstruction was produced in under a second.}
\label{fig:nufft}
\end{figure}

Our reconstruction may be compared to one achieved with CG.  Again time-limiting the CG reconstruction based on the runtime of the tangential interpolator, we computed the iterative solution to the system.  With both solutions available, we obtained the residual vector $x-\hat{x}$ for each method and plotted the results on the same scale in Fig.~\ref{fig:nufft2}.  While both methods achieve reasonable reconstructions, our tangential-interpolation algorithm outperforms CG in reconstruction quality for an equal amount of computation time.

\begin{figure}[!t]
\centering
\psfrag{Residual Vectors}{\fontsize{10}{12}\selectfont \hspace{-30pt}{\bf Reconstruction Residual Vectors}}
\psfrag{Sample Index}{\raisebox{-3pt}{\fontsize{9}{10.8}\selectfont \hspace{0pt}Sample Index}}
\psfrag{Residual Value}{\raisebox{1pt}{\fontsize{9}{10.8}\selectfont \hspace{-5pt}Residual Value}}
\psfrag{CG Residual}{\fontsize{7.5}{9}\selectfont CG Residual}
\psfrag{Tan. Int. Residual}{\fontsize{7.5}{9}\selectfont Tan. Int. Residual}
\includegraphics[width=4.5in]{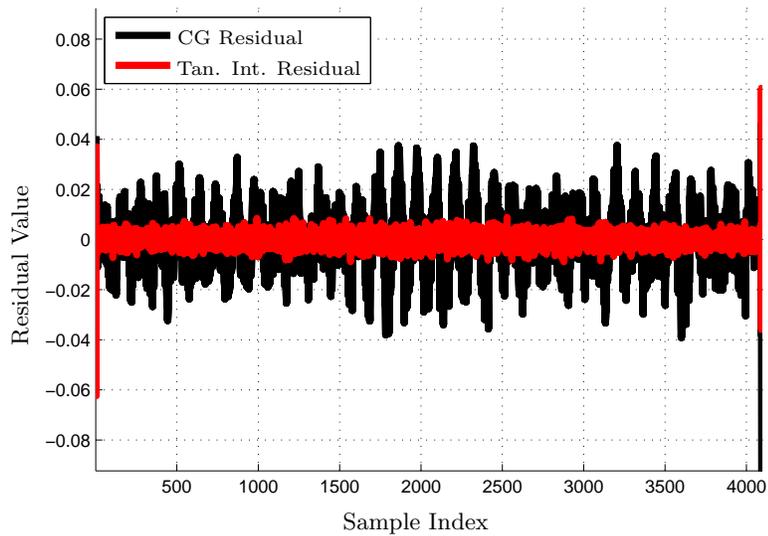}
\caption{\small\sl Residuals of the reconstructed signals from our tangential-interpolation method and the CG method.  Both methods yield a reasonable reconstruction in the same amount of time, with the tangential-interpolation method having a smaller residual.}
\label{fig:nufft2}
\end{figure}

\section{Summary and future extensions}
\label{sec:conclusions}

In this work, we have proposed an algorithm for solving three different Toeplitz-structured Tikhonov-regularization problems with complexity $\mathcal{O}(N\log^2 N)$, where $N$ is the total number of free parameters.  This algorithm solves a tangential-interpolation problem in place of a linear-algebraic problem, much like the superfast pseudoinversion algorithm of~\cite{Barel:2003}, and is based on the ``extension-and-transformation'' approach of~\cite{Heinig:1996}. Further, it is stabilized and non-iterative.

We have demonstrated through a series of experiments that our implementation of the proposed algorithm produces accurate results and does so with a computational cost that closely matches the theoretical asymptotic bound.  While a direct comparison between the two algorithms is difficult, we were able to frame our results in terms of the number of iterations that a CG solver may perform in the same amount of time. In comparing the accuracy of the two solution methods for the same amount of time, our algorithm appears to be preferable across a wide range of scenarios. In addition, we showed that our algorithm may be used in practical settings by recovering a signal from its non-uniform spectral samples.

It is possible to extend our approach to {\sl multilevel} Toeplitz matrices, which exhibit Toeplitz structure in multiple scales. The extension-and-transformation approach can be applied directly to such problems, but unfortunately requires a much higher asymptotic cost than desired.  However, we have recently presented an efficient inversion algorithm for two-level Toeplitz matrices that exhibit triangularity in one or more of their levels~\cite{Turnes:2012}.  In a future work, we will combine the theoretical results that form the basis of the inversion method with the regularization algorithm of this paper to produce a least-squares solver for this class of matrices.

\bibliographystyle{IEEEbib}
\bibliography{Tikhonov}

\end{document}